\documentclass{amsproc}

\newcommand{\Span}[1]{\left<#1\right>}
\newcommand{\Hilb}[2]{#1\text{-Hilb}(#2)}
\newcommand{\plus}{(\!+\!)}
\newcommand{\minus}{(\!-\!)}
\newcommand{\Mtheta}{\mathcal{M}_{d,\theta}(Q,R)}
\newcommand{\C}{\mathbb C}

\newcommand{\Z}{\mathbb Z}

\newcommand{\PP}{\mathbb P}
\newcommand{\LL}{\mathcal L}

\newcommand{\FF}{\mathcal F}

\newcommand{\RR}{\mathcal R}

\newcommand{\ZZ}{\mathcal Z}

\newcommand{\OO}{\mathcal O}
\DeclareMathOperator{\GL}{GL}
\DeclareMathOperator{\SL}{SL}

\DeclareMathOperator{\Irr}{Irr}

\DeclareMathOperator{\Hom}{Hom}

\DeclareMathOperator{\diag}{diag}

\DeclareMathOperator{\D}{D}

\DeclareMathOperator{\SO}{SO}

\DeclareMathOperator{\End}{End}

\DeclareMathOperator{\Pic}{Pic}
\DeclareMathOperator{\rk}{rank}

\usepackage{verbatim}
\usepackage{amssymb}
\usepackage{amsbsy}
\usepackage{amscd}
\usepackage{amsmath}
\usepackage{amsthm}
\usepackage{hyperref}
\usepackage[mathscr]{eucal}

\usepackage{tikz,mathrsfs}
\usetikzlibrary{arrows,decorations.pathmorphing,decorations.pathreplacing,positioning,shapes.geometric,shapes.misc,decorations.markings,decorations.fractals,calc,patterns}
\tikzset{vertex/.style={circle,fill=black,inner sep=1pt,outer sep=3pt},
		gap/.style={inner sep=0.5pt,fill=white}}
\def\hex#1#2{
  	\draw #1 #2 +(30:0.7) \foreach \a in {90,150,210,270,330} { -- +(\a:0.7) } -- cycle;
	}

\def\r{1.212435}

\newtheorem{theorem}{Theorem}
\numberwithin{defn}{section}
\newtheorem{prop}{Proposition}\numberwithin{defn}{section}
\numberwithin{defn}{section}
\newtheorem{exa}{Example}\numberwithin{exa}{section}
\newtheorem{rem}{Remark}\numberwithin{exa}{section}

\begin{document}

\title[]{On Reid's recipe for non abelian groups}

\author[A. Nolla de Celis]{\'Alvaro Nolla de Celis}
\address{Departamento de Did\'acticas Espec\'ificas. \\Facultad de Formaci\'on del Profesorado y Educaci\'on. \\Universidad Aut\'onoma de Madrid. \\Tom\'as y Valiente, 3. \\28049, Madrid.}
\email{alvaro.nolla@uam.es}

\subjclass[2010]{}

\begin{abstract}
This paper presents two new explicit examples of Reid's recipe for non-abelian groups in $\SL(3,\C)$, namely the dihedral group $\mathbb{D}_{5,2}$ and a trihedral group of order 39.
\end{abstract}

\maketitle

\section{Introduction}

If $G$ is a finite subgroup of $\SL(2,\C)$, the {\em classical McKay correspondence} establish a one to one bijection between non-trivial irreducible representations of $G$ and the irreducible components of the exceptional locus $E$ in the minimal resolution $Y$ of the du Val singularity $\C^2/G$. It is well known that in this case the exceptional locus $E\subset Y$ consists of a chain of rational curves $E_i\cong\PP^1$ with $E_i^2=-2$ and intersection graph of ADE type. 

On the other hand, McKay \cite{McKay} observed that if we take the set of irreducible representations of $G$ to be $\Irr G=\{\rho_0,\rho_1,\ldots,\rho_n\}$ and $V:G\to\SL(2,\C)$ the natural representation, then the {\em McKay graph} formed by a vertex for each $\rho\in\Irr G$ and $a_{ij}$ edges between $\rho_i$ and $\rho_j$ such that $V\otimes\rho_i=\sum_{j=0}^na_{ij}\rho_j$, has the same (extended) ADE type. Thus, the correspondence can be stated as the bijection: 
\begin{equation}\label{McKay}
\text{$\{$Irreducible components $E_i\subset Y \} \longleftrightarrow \{$Non-trivial $\rho_i\in\Irr G\}$}
\end{equation}

\begin{exa} Let $G=D_4=\Span{\alpha=\left(\begin{smallmatrix}i&0\\0&-i\end{smallmatrix}\right), \beta=\left(\begin{smallmatrix}0&1\\-1&0\end{smallmatrix}\right)}$ the binary dihedral group of order 8 in $\SL(2,\C)$. The group has five irreducible representations: four 1-dimensional $\rho_0^+$ (trivial), $\rho_0^-$, $\rho_2^+$, $\rho_2^-$; and the 2-dimensional representation $\rho_1\cong V$. The minimal resolution $\pi:Y\to\C^2/G$ consists of four rational curves with intersection graph of the exceptional locus $E=\pi^{-1}(0)$ shown in Figure \ref{exa:D4:group} (left). 

On the other hand, the McKay graph obtained in $\Irr G$ is of ADE type (extended) $\overline{D}_4$, with coincides with the resolution graph if we don't consider the trivial representation. See Figure \ref{exa:D4:group} (right).

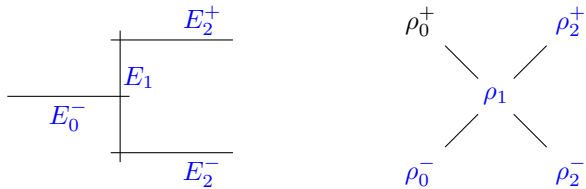
\begin{figure}[htbp]
\begin{center}
\begin{tikzpicture}

\node at (0,0) {
\begin{tikzpicture}[scale=0.5]

\draw (0,1.5) to 
	node[below, inner sep=1.5pt, blue] {$E_0^-$}(3.25,1.5);
\draw (3,-0.25) to 
	node[below, pos=0.6, inner sep=10pt, blue] {} (3,3.25);
\node[blue] at (3.5,2) {$E_1$};
\draw (2.75,3) to 
	node[above, near end, inner sep=1.5pt, blue] {$E_2^+$} (6,3);
\draw (2.75,0) to 
	node[below, near end, inner sep=1.5pt, blue] {$E_2^-$} (6,0);
\end{tikzpicture}
};

\node at (5,0)
{\begin{tikzpicture}
 \node (0u) at (0,1) {$\rho_0^+$};
 \node[blue] (0d) at (0,-1) {$\rho_0^-$};
 \node[blue] (1) at (1,0) {$\rho_1$};
 \node[blue] (2u) at (2,1) {$\rho_2^+$};
 \node[blue] (2d) at (2,-1) {$\rho_2^-$};
\draw [-] (0u) --  (1);
\draw [-] (0d) -- (1);
\draw [-] (1) -- (2u);
\draw [-] (1) -- (2d);
\end{tikzpicture}};

\end{tikzpicture}
\vspace{-0.25cm}
\caption{Graph of the exceptional locus $E\subset Y$ (left) and the McKay graph for $D_4\subset\SL(2,\C)$ of (extended) type $D_4$.}
\label{exa:D4:group}
\end{center}
\end{figure}

\end{exa}

\begin{rem} The choice of notation $\rho_0^\pm$ and $\rho_2^\pm$ comes from the fact that irreducible representations of $D_4$ can be induced from the irreducible representations of the cyclic group $A_4=\Span{\alpha}$. In this group $\Irr A_4=\{\rho_0,\rho_1,\rho_2,\rho_3\}$ where $\rho_i(\alpha)=i^k$ for $i=1,\ldots,3$. The action of $D_4/A_4\cong\Z/2\Z$ on $\Irr A_4$ fixes $\rho_0$ and $\rho_2$, producing $\rho_0^\pm$ and $\rho_2^\pm$, and the orbit $\{\rho_1,\rho_3\}$ produces the 2-dimensional representation $\rho_1$. Similar situation and notation is present in the examples of Sections \ref{exa:Dn} and \ref{exa:T}.
\end{rem}

In this sense, we can assign a non-trivial irreducible representation $\rho\in\Irr G$ to every component $E_\rho\subset Y$ in the minimal resolution, so the McKay correspondence reflects how the geometry of $Y$ is determined by the representation theory of $G$.  

The geometric interpretation of the correspondence was done by Gonzalez-Springer and Verdier in \cite{GSV} as follows: for each nontrivial $\rho_i\in\Irr G$ there is a vector bundle $\RR_{i}$ on $Y$ such that 
\begin{equation}\label{c1GSV}
\deg(c_1(\RR_{i}|_{E_j}))=\delta_{ij}, \hspace{1cm} j=1,\ldots,n.
\end{equation}

The bundles $\RR_i$ can be constructed in the following way. From the work of Ito and Nakamura \cite{IN} it is know that the $G$-equivariant Hilbert scheme $\Hilb{G}{\C^2}$, the fine moduli space of scheme-theoretic orbits or {\em $G$-clusters}, is the minimal resolution $Y$. Recall that a {\em $G$-cluster} is a $G$-invariant 0-dimensional subscheme $\ZZ$ such that $\OO_\ZZ\cong\C[G]$ the regular representation as $G$-modules. The interpretation of $Y\cong\Hilb{G}{\C^2}$ as a moduli space allows us to consider the tautological bundle $\RR$ on $Y$, where the fibre over a point $y\in Y$ is $H^0(Z(y),\OO_{Z(y)})$ where $Z(y)$ is the orbit parametrised by $y$. By definition of $G$-cluster $H^0(Z(y),\OO_{Z(y)})\cong\C[G]\cong\bigoplus_{\rho\in\Irr G}\rho^{\dim\rho}$. This fact imply that $\RR$ decomposes as a direct sum of globally generated vector bundles: 
\[
\RR=\bigoplus_{i}\RR_i\otimes\rho_i
\]
where $\RR_i=\Hom_G(\rho_k,\RR)$ are the {\em tautological bundles} of $Y$.

Notice that $\RR_i\in\Pic Y$, and $\RR_{0}\cong\OO_Y$ generates $H^0(Y,\Z)$. Also, since $E_i$ for $i>0$ form a basis of $H_2(Y,\Z)$, dual to a basis of $H^2(Y,\Z)$, we can use the Gonzalez-Springer and Verdier result to state the {\em integral McKay correspondence} as the bijection: 
\begin{equation}\label{intMcKay}
\text{$\{$Irreducible representations of $G\}\longleftrightarrow$ basis of $H^*(Y,\Z)$}
\end{equation}

\begin{exa}\label{exa:D4} We now illustrate explicitly the correspondence (\ref{intMcKay}) continuing with the group $G=D_4$, and we introduce some notation that will be present in following examples.

By \cite{INj} it is known that we can interpret $\Hilb{D_4}{\C^2}$ as the moduli space $\Mtheta$ of $\theta$-stable representations of dimension vector $d=(\dim\rho_i)_{\rho_i\in\Irr D_4}=(1,1,1,1,2)$ of the McKay quiver $Q$ with suitable relations $R$ (see Figure \ref{exa:D4:quiver}), for a particular choice of generic parameter $\theta^0$. 

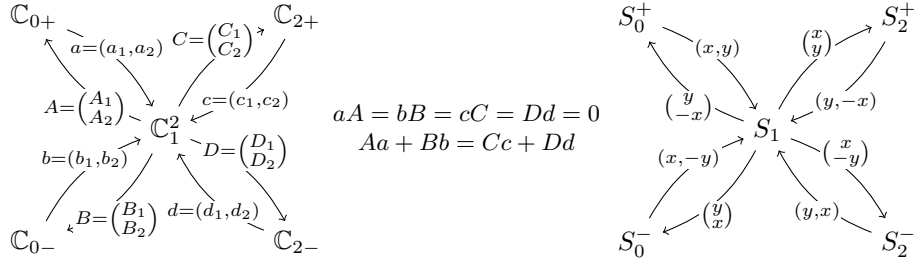
\begin{figure}[htbp]
\begin{center}
\begin{tikzpicture}
\node (n1) at (0,0) 
{
\begin{tikzpicture} [bend angle=45, looseness=1]
\node (0m) at (0,0)  {$\small \C_{0-}$};
\node (0p) at (0,3)  {$\small \C_{0+}$};
\node (2p) at (3.5,3)  {$\small \C_{2+}$};
\node (2m) at (3.5,0)  {$\small \C_{2-}$};
\node (1) at (1.75,1.5) {$\small \C^2_{1}$};
\draw [->,bend left=20] (0p) to node[above,gap]  {$\scriptstyle a=(a_1,a_2)$} (1);
\draw [->,bend left=20] (1) to node[below,gap]{$\scriptstyle A=\left(\!\begin{smallmatrix}A_1\\A_2\end{smallmatrix}\!\right)$} (0p);
\draw [->,bend left=20] (0m) to node[above,gap]  {$\scriptstyle b=(b_1,b_2)$} (1);
\draw [->,bend left=20] (1) to node[below,gap]{$\scriptstyle B=\left(\!\begin{smallmatrix}B_1\\B_2\end{smallmatrix}\!\right)$} (0m);
\draw [->,bend left=20] (2p) to node[below,gap]  {$\scriptstyle c=(c_1,c_2)$} (1);
\draw [->,bend left=20] (1) to node[above,gap]{$\scriptstyle C=\left(\!\begin{smallmatrix}C_1\\C_2\end{smallmatrix}\!\right)$} (2p);
\draw [->,bend left=20] (2m) to node[below,gap]  {$\scriptstyle d=(d_1,d_2)$} (1);
\draw [->,bend left=20] (1) to node[above,gap]{$\scriptstyle D=\left(\!\begin{smallmatrix}D_1\\D_2\end{smallmatrix}\!\right)$} (2m);
\end{tikzpicture}
};
\node (n2) at (4,0) {\small$
\begin{array}{c}
aA=bB=cC=Dd=0 \\
Aa+Bb=Cc+Dd\\
\end{array}
$
};

\node (n1) at (8,0) 
{
\begin{tikzpicture} [bend angle=45, looseness=1]
\node (0m) at (0,0)  {$\small S_0^-$};
\node (0p) at (0,3)  {$\small S_0^+$};
\node (2p) at (3.5,3)  {$\small S_2^+$};
\node (2m) at (3.5,0)  {$\small S_2^-$};
\node (1) at (1.75,1.5) {$\small S_1$};
\draw [->,bend left=20] (0p) to node[above,gap]  {$\scriptstyle (x,y)$} (1);
\draw [->,bend left=20] (1) to node[below,gap]{$\scriptstyle \left(\!\begin{smallmatrix}y\\-x\end{smallmatrix}\!\right)$} (0p);
\draw [->,bend left=20] (0m) to node[above,gap]  {$\scriptstyle (x,-y)$} (1);
\draw [->,bend left=20] (1) to node[below,gap]{$\scriptstyle \left(\!\begin{smallmatrix}y\\x\end{smallmatrix}\!\right)$} (0m);
\draw [->,bend left=20] (2p) to node[below,gap]  {$\scriptstyle (y,-x)$} (1);
\draw [->,bend left=20] (1) to node[above,gap]{$\scriptstyle \left(\!\begin{smallmatrix}x\\y\end{smallmatrix}\!\right)$} (2p);
\draw [->,bend left=20] (2m) to node[below,gap]  {$\scriptstyle (y,x)$} (1);
\draw [->,bend left=20] (1) to node[above,gap]{$\scriptstyle \left(\!\begin{smallmatrix}x\\-y\end{smallmatrix}\!\right)$} (2m);
\end{tikzpicture}
};
\end{tikzpicture}
\vspace{-0.25cm}
\caption{Representation space (left) with relations and the maps between the coinvariant algebras $S_\rho$ (right) for $D_4$.}
\label{exa:D4:quiver}
\end{center}
\end{figure}

To calculate explicitly $\Mtheta$ we construct the representation space by putting a vector space $\C^{\dim(\rho)}_\rho$ of dimension $\dim(\rho)$ at each vertex $\rho$ in $Q$, and linear maps between them according to the arrows of $Q$ (see Figure \ref{exa:D4:quiver} left). Then we choose $\theta^0$ to be the 0-generated stability condition, i.e. $\theta^0$ is generic and $\theta^0_i>0$ for every $i\neq0$, which in practice means that for every vertex $\C^{\dim(\rho)}_\rho$ with $\rho\in\Irr D_4\backslash\{\rho_0^+\}$ there exists $\dim(\rho)$ linearly independent paths from $\C_{\rho_0^+}$ to $\C^{\dim(\rho)}_{\rho}$. The computations for $D_4$ gives us $\Hilb{D_4}{\C^2}$ as the union of five open sets $U_i$, $i=1,\ldots,5$, all of them isomorphic to a hypersurface $f=0$ in $\C^3$. See Table \ref{exa:D4:opens} for the choices of linear independent paths. 

\begin{table}[htp]
\begin{center}
\renewcommand{\arraystretch}{1.5}
\begin{tabular}{ccc}
 Open	& Linear independent paths & $(f=0)\subset\C^3$ \\
\hline
$U_1$ & $a=(1,0)$, $c=(0,1)$, $B_2=C_1=D_1=1$ & $D_2(B_1^2b_1+1)=B_1b_1$  \\
$U_2$ & $a=(1,0)$, $c=(0,1)$, $B_1=C_1=D_1=1$ & $B_2(d_2+1)=D_2d_2$  \\
$U_3$ & $a=(1,0)$, $d=(0,1)$, $B_1=C_1=D_1=1$ & $C_2(b_2-1)=B_2b_2$  \\
$U_4$ & $a=(1,0)$, $c=(0,1)$, $B_1=C_1=D_2=1$ & $B_2(D_1^2d_1-1)=D_1d_1$ \\
$U_5$ & $a=(1,0)$, $d=(0,1)$, $B_1=C_2=D_1=1$ & $B_2(C_1^2c_1-1)=C_1c_1$ \\
\hline
\end{tabular}
\end{center}
\caption{Open cover $\Hilb{D_4}{\C^2}=\bigcup_{i=1}^5 U_i$.}
\label{exa:D4:opens}
\end{table}%

To write down the basis for the tautological bundles $\RR_\rho$ we need to consider the {\em coinvariant algebras} $S_\rho$ for every $\rho\in\Irr G$, which are the Cohen-Macaulay $\C[x,y]^G$-modules $S_\rho:=(\C[x,y]\otimes_\C\rho)^G$ where $G$ acts on $\C[x,y]$ by the inverse transpose. For example, the polynomial $(+):=x^2+y^2$ belongs to $S_{\rho_2^+}$ because:
\begin{align*}
\alpha(x^2+y^2\otimes e_2^+) &= (-ix)^2+(ix)^2 \otimes e_2^+ = x^2+y^2\otimes e_2^+
\\
\beta(x^2+y^2\otimes e_2^+) &= (-y)^2+x^2 \otimes e_2^+ = x^2+y^2\otimes e_2^+
\end{align*}
where we denoted with $e_2^+$ the basis of $\rho_2^+$. Similarly, the polynomial $\minus:=x^2-y^2\in S_{\rho_2^+}$, and $S_{\rho_1}$ contains the pairs $(x,y)$ and $(y\minus, x\minus)$. Clearly $S_{\rho_0^+}=\C[x,y]^{D_4}$.

It is well know that if $G\subset\SL(n,\C)$ is finite, then the McKay quiver $Q$ with (suitable) relations is isomorphic to the skew group algebra $\C[x,y]\ast G$, which is isomorphic to the endomorphism algebra $\End_{\C[x,y]^G}(\bigoplus_{\rho\in\Irr G}S_\rho^{\dim(\rho)})$ by a result of Auslader. This allows us to use the linearly independent arrows at every open set to choose the basis of every tautological bundle $\RR_\rho$. In our example see Figure \ref{exa:D4:quiver} (right) for the quiver among the coinvariant algebras for $G=D_4$. The linearly independent paths of Table \ref{exa:D4:opens} produce the basis of the tautological bundles $\RR_\rho$ at every open set in $\Hilb{D_4}{\C^2}$ shown in Table \ref{exa:D4:taut}. 

\begin{table}[htp]
\begin{center}
\setlength{\tabcolsep}{3pt}
\renewcommand{\arraystretch}{1.5}
\begin{tabular}{ccccccc}
 	& $\RR^+_0$ & $\RR^-_0$ & $\RR^+_2$ & $\RR^-_2$ & $\RR_{1}$ & $\det\RR_1$\\
\hline
$U_1$ & $1$ & $\plus\minus$ & $\plus$ & $\minus$ & $\left(\begin{smallmatrix}x\\y\end{smallmatrix}\right)$, $\left(\!\begin{smallmatrix}y\plus\\-x\plus\end{smallmatrix}\!\right)$ & $\plus^2$ \\
$U_2$ & $1$ & $xy$ & $\plus$ & $\minus$ & $\left(\begin{smallmatrix}x\\y\end{smallmatrix}\right)$, $\left(\!\begin{smallmatrix}y\plus\\-x\plus\end{smallmatrix}\!\right)$ & $\plus^2$  \\
$U_3$ & $1$ & $xy$ & $\plus$ & $\minus$ & $\left(\begin{smallmatrix}x\\y\end{smallmatrix}\right)$, $\left(\!\begin{smallmatrix}y\minus\\x\minus\end{smallmatrix}\!\right)$ & $\minus^2$  \\
$U_4$ & $1$ & $xy$ & $\plus$ & $xy\plus$ & $\left(\begin{smallmatrix}x\\y\end{smallmatrix}\right)$, $\left(\!\begin{smallmatrix}y\plus\\-x\plus\end{smallmatrix}\!\right)$ & $\plus^2$ \\
$U_5$ & $1$ & $xy$ & $xy\minus$ & $\minus$ & $\left(\begin{smallmatrix}x\\y\end{smallmatrix}\right)$, $\left(\!\begin{smallmatrix}y\minus\\x\minus\end{smallmatrix}\!\right)$ & $\minus^2$  \\
\hline
\end{tabular}
\end{center}
\caption{Basis of sections of each $\RR_\rho$ on every open set in $\Hilb{D_4}{\C^2}$. Notation: $\plus:=x^2+y^2$ and $\minus:=x^2-y^2$.}
\label{exa:D4:taut}
\end{table}

Using the sections in Table \ref{exa:D4:taut} we can check that $\RR_{\rho}$ is $\OO(1)$ restricted to $E_\rho$ and trivial at the rest of curves, so equation (\ref{c1GSV}) is verified. In this way we can ``mark" every curve in $E$ with an irreducible representation of $G$. Observe that if the curve is marked with a 1-dimensional representation $\rho$, then the  sections of $\RR_\rho$ define the local coordinates of the $E_\rho\cong\PP^1$. For the curve $E_1$, which is marked with the 2-dimensional representation $\rho_1$, the ratio is defined by the relation $d_2(y\plus,-x\plus)=b_2(y\minus,x\minus)$. The explicit description of the exceptional locus in $\Hilb{D_4}{\C^2}$ and the corresponding coordinates and representations is shown in Figure \ref{exa:D4:McKay}.

\begin{figure}[htbp]
\begin{center}
\begin{tikzpicture}

\draw (0,1.5) to 
	node[black, above, inner sep=1.5pt] {\small $xy:\plus\minus$} 
	node[below, inner sep=1.5pt, blue] {$\rho_0^-$}(3.25,1.5);
\draw[->] (0,1.5) to node[above, inner sep=1.5pt] {\scriptsize $B_1$} (0.5,1.5);
\draw (3,-0.25) to node[below, pos=0.6, inner sep=10pt, blue] {} (3,3.25);
\draw[->] (3,1.5) to node[above, inner sep=1.5pt] {\scriptsize $B_2$} (2.5,1.5);
\node[blue] at (3.25,2) {$\rho_1$};

\draw (2.75,0) to 
	node[above, inner sep=1.5pt] {\small $\minus:2xy\plus$}
	node[below, inner sep=1.5pt, blue] {$\rho_2^-$} (6,0);

\draw[->] (3,3) to node[left, inner sep=1.5pt] {\scriptsize $b_2$} (3,2.5);	
\draw[->] (3,3) to node[above, inner sep=1.5pt] {\scriptsize $C_2$} (3.5,3);	

\draw (2.75,3) to 
	node[above, inner sep=1.5pt] {\small $\plus:2xy\minus$}
	node[below, inner sep=1.5pt, blue] {$\rho_2^+$} (6,3);

\draw[->] (3,0) to node[left, inner sep=1.5pt] {\scriptsize $d_2$} (3,0.5);	
\draw[->] (3,0) to node[below, inner sep=1.5pt] {\scriptsize $D_2$} (3.5,0);	

\draw[->] (6,0) to node[below, inner sep=1.5pt] {\scriptsize $D_1$} (5.5,0);	
\draw[->] (6,3) to node[above, inner sep=1.5pt] {\scriptsize $C_1$} (5.5,3);	

\node[red] at (-0.25,1.5) {\Large $U_1$};
\node[red] at (2.5,-0.25) {\Large $U_2$};
\node[red] at (2.5,3.25) {\Large $U_3$};
\node[red] at (6.35,-0.25) {\Large $U_4$};
\node[red] at (6.35,3.25) {\Large $U_5$};

\end{tikzpicture}
\vspace{-0.2cm}
\caption{McKay correspondence for the group $D_4\subset\SL(2,\C)$.}
\label{exa:D4:McKay}
\end{center}
\end{figure}
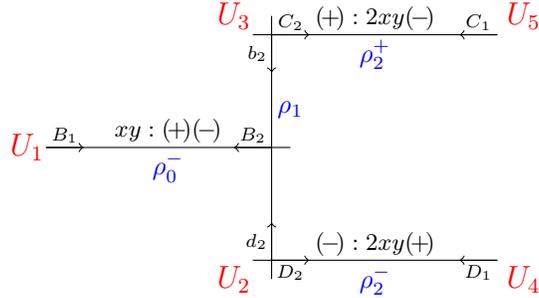

\end{exa}

Reid \cite{Reid} conjectured that statement (\ref{intMcKay}) should hold for any finite subgroup $G\subset\SL(n,\C)$ as long as a crepant resolution $\pi:Y\to\C^n/G$ exists. Under this circumstances, the (tautological) bundles $\RR_\rho$  form a $\Z$-basis of the $K$-theory of $Y$, and certain cookery with the Chern classes $c_i(\RR_\rho)$ would lead to a $\Z$-basis  of $H^*(Y,\Z)$.

When $n=3$ it is known that there always exists such a resolution. In fact, by \cite{BKR} it is know that $Y=\Hilb{G}{\C^3}$ is a crepant resolution of $\C^3/G$ for any finite subgroup of $\SL(3,\C)$. By definition, a $G$-cluster contains all the information of the representation theory of $G$, so even though crepant resolutions are no longer unique the $\Hilb{G}{\C^3}$ is in this sense ``distinguished", and becomes a suitable resolution to test the conjecture. 

For $G=A\subset\SL(3,\C)$ an abelian subgroup, Reid \cite{Reid} verified the correspondence (\ref{intMcKay}) with a series of examples defining a decoration or {\em marking} of the junior simplex of the toric resolution $Y=\Hilb{A}{\C^3}$ which leads to (i) the relations between the tautological bundles $\RR_\rho$ in $\Pic Y$ inducing a basis for $H^2(Y,\Z)$, and (ii) the definition of certain virtual bundles $\FF_E$ for each compact surface $E\subset Y$ whose second Chern class $c_2(\FF_E)$ form a basis of $H^4(Y,\Z)$. The marking and cookery is the so called {\em Reid's recipe}. The extension to every abelian subgroup in $SL(3,\C)$ was completed by Craw \cite{C}, which can be stated as follows:

\begin{theorem}[\cite{C}] The conjecture \ref{intMcKay} holds for any finite abelian subgroup $A\subset\SL(3,\C)$ and $Y=\Hilb{A}{\C^3}$. 
\end{theorem}

The correspondence can be described precisely as follows:
\begin{itemize}
\item $\RR_{\rho_0}\cong\OO_Y$ generates $H^0(Y,\Z)$.
\item $\{[c_1(\RR_\rho)]\}$ such that $\rho$ either marks a curve in $Y$ or is the second representation marking a del Pezzo surface $dP_6\subset Y$ forms a basis of $H^2(Y,\Z)$.
\item $\{[c_2(\FF_E)]\}$ such that $E\subset Y$ is a compact surface forms a basis of $H^4(Y,\Z)$.
\end{itemize}

See Example \ref{exa:Ab} for explicit computations.

Using the language of derived categories, a generalized version of the correspondence was introduced by Cautis-Logvinenko \cite{CL} and completed in Cautis-Craw-Logvinenko \cite{CCL} for $A\subset\SL(3,\C)$, giving rise to the so called {\em derived Reid's recipe}. This interpretation also remains open for the non-abelian case. Finally, for the dimer model case, which covers every three-dimensional toric Gorenstein singularity, the recipe has been studied in \cite{BCQ} and \cite{CHT}.

There are very few examples of this explicit correspondence when the group $G$ is non-abelian, mainly due to the lack of examples of explicit descriptions of crepant resolutions away from the toric case. One can find the case of finite $G\subset\SO(3)$ in \cite{NS}, where $\Hilb{\C^3}{G}$ is computed. In this case the exceptional fibre $\pi^{-1}(0)$ is 1-dimensional, so the exceptional locus does not contain compact divisors, and there is a one-to-one correspondence between (non-trivial) irreducible exceptional curves in $E$ and (non-trivial) irreducible representations of $G$. With the presence of compact divisors, Gomi-Nakamura-Shinoda \cite{GNS} computed the exceptional locus of $\Hilb{G}{\C^3}$ for $G_{60}$ and $G_{168}$, the only two non-abelian simple groups in $\SL(3,\C)$. The resolution of $\C^3/G_{60}$ has 1-dimensional fiber over $\pi^{-1}(0)$ whereas for $G_{168}$ consists of a smooth rational curve and a doubly blown-up projective plane intersecting transversally at a unique point.

In this paper we include the explicit description for two new non-abelian examples in $\SL(3,\C)$, namely the dihedral group $\mathbb{D}_{5,2}$ (Section \ref{exa:Dn}) and the trihedral group $\Span{\frac{1}{13}(1,3,9),T}$ (Section \ref{exa:T}), in order to define an analogue of Reid's recipe and check the correspondence (\ref{intMcKay}) on them. Although there are some differences from the abelian case, most of the features and spirit of the recipe seem to hold. Both examples are contained in two interesting families of subgroups of $\SL(3,\C)$, and may constitute the first step to generalize Reid's recipe to this cases. 

I would like to thank the organisers of the conference {\em The McKay Correspondence, Mutations and Related Topics} for hosting the meeting under difficult circumstances, R. Mu\~noz and A. Craw for lots of useful discussions, and grateful to A. Craw for sharing with me an earlier version of \cite{C2}. I owe special thanks to M. Reid for introducing me to these beautiful families of examples and for so many insights into their absorbing calculations.

\section{Marking the exceptional locus}\label{sec:marking}

The {\em marking} of the irreducible components of the exceptional locus $E\subset Y=\Hilb{G}{\C^3}$ for $G\subset\SL(3,\C)$ finite is done as follows. Let $\{U_{i}\}_{i=0}^{m}$ be an open cover of $Y$, and denote by $r_{i,\rho}$ the basis of sections of $\RR_\rho$ at $U_i$ (there are $\dim(\rho)$ elements in $r_{i,\rho}$). Note that $r_{i,\rho}\in S_\rho$. Then $\ZZ_i=\{r_{i,\rho}~|~\rho\in\Irr G\}$ form a basis for any $G$-cluster in $U_i$, also known as a {\em $G$-graph}.

Let $C\subset Y$ be a rational curve. Define the submodule $G$-igMod$_C(U_i):=\{r_{i,\rho}~|~\deg(\RR_i|_C)\neq0\}$ of $\ZZ_i$. This submodule contains the {\em $G$-igsaw transformation} along the curve $C$ among the basis elements in $\RR$. Now, the curve $C\subset Y$ is covered by two open sets, $C\subset U_{i}\cap U_{j}$ and we say that
\[
\text{$\rho$ marks $C \iff r_{\alpha,\rho}$ generates $G$-igMod$_C(U_\alpha)$ for $\alpha=i$ or $\alpha=j$}
\]
Let $D\subset Y$ be a compact divisor, and let $\{U_i\}_{i=1}^k$ be an open covering of $D$. Then
\[
\begin{array}{ccc}
\text{$\rho$ marks $D$} & \iff & \text{$r_{i,\rho}$ is in the socle of $\ZZ_i$ at the } \\&& \text{origin of $U_i$, for $i=1,\ldots,k$}
\end{array}
\]

\begin{rem} Notice that it may happen that a curve $C$ is marked by two representations (see the curve $C_1$ in Example \ref{exa:T}), as it already happens in the dimer model case (e.g. \cite[4.19]{CHT}) . This fact contrasts with the toric case where the submodules $G$-igMod$_C(U_i)$ and $G$-igMod$_C(U_j)$ are always generated from the same representation. Also, a divisor $D$ may be marked with more than one representation (see divisor $E_1$ in the trihedral example of Section \ref{exa:T}), a situation present in the abelian case when $Y$ contains a del Pezzo $dP_6$ surface. 
\end{rem}

For the abelian case $G=A$, this definition is equivalent to the marking initially proposed by the Reid's recipe, which takes advantage of the toric structure of $\Hilb{A}{\C^3}$ (see \cite[$\S$3, $\S$6.2]{C}). Both definitions can be compared in the following axample.

\begin{exa}\label{exa:Ab}
Let $A$ be the following abelian group of order 12:
\[
A=\frac{1}{12}(1,7,4):=\Span{\alpha=\diag(\varepsilon,\varepsilon^7,\varepsilon^4)~|~\varepsilon=e^{2\pi i/12}}\subset\SL(3,\C)
\]
which has 12 irreducible representations $\rho_i:A\to\C$ verifying $\rho_i(\alpha)=\varepsilon^i$ for $i=0,\ldots,11$. The action of $A$ on $\C[x,y,z]$ decomposes into the coinvariant algebras $S_i:=S_{\rho_i}$, $i=0,\ldots,11$. Some of the elements in these algebras are shown in Table \ref{ex:Ab:Srho} (omitting the $_-\otimes e_i$ by abusing the notation). The group is abelian so every tautological bundle $\RR_i:=\RR_{\rho_i}$ has rank 1.

\begin{table}[h]
\begin{center}
\begin{tabular}{|c|l||c|l||c|l|}
\hline
$S_0$ & 1, $x^{12}$, $y^{12}$, $z^3$ & $S_4$ & $x^4$, $y^4$, $z$ & $S_8$ & $x^{8}$, $y^{8}$, $z^2$, $xy$ \\
$S_1$ & $x$, $y^{7}$, $y^3z$ & $S_5$ & $x^5$, $y^{11}$, $xz$, $xy^4$ & $S_9$ & $x^{9}$, $y^{3}$, $x^2y$, $xz^2$ \\
$S_2$ & $x^{2}$, $y^{2}$ & $S_6$ & $x^6$, $y^6$, $x^2z$, $y^2z$ & $S_{10}$ & $x^{10}$, $y^{10}$, $x^3y$, $xy^3$ \\
$S_3$ & $x^{3}$, $y^{9}$, $xy^2$, $yz^2$ & $S_7$ & $x^7$, $y$, $x^3z$ & $S_{11}$ & $x^{11}$, $y^{5}$, $yz$, $x^4y$ \\
\hline
\end{tabular}
\end{center}
\caption{Coinvariant algebras $S_i$ for the group $A=\frac{1}{12}(1,7,4)$.}
\label{ex:Ab:Srho}
\end{table}%

Then $Y=\Hilb{A}{\C^3}\to\C^3/A$ is a toric crepant resolution that is described using Craw-Reid \cite{CR} by the triangulation of the junior simplex shown in Figure \ref{exa:Ab:marking} (left). Every triangle corresponds to an affine open set isomorphic to $\C^3$, every interior edge corresponds to a rational curve $C\subset Y$ parametrized by the ratios indicated in the figure, and every interior point corresponds to an irreducible compact divisor in $Y$.   

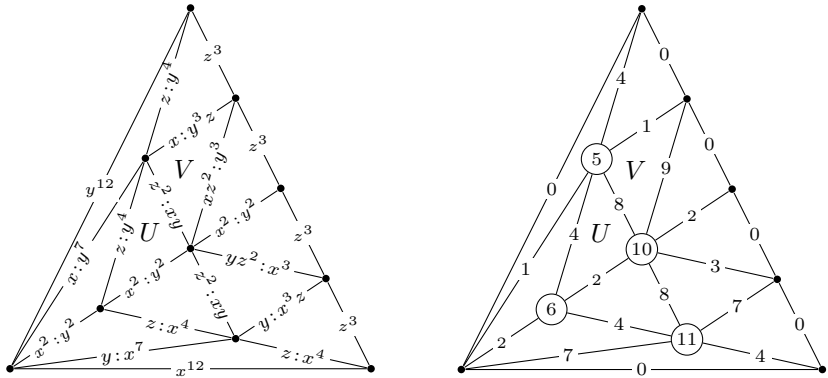
\begin{figure}[htbp]
\begin{center}
\begin{tikzpicture}
\node at (0,0) 
{
\begin{tikzpicture}[scale=0.4]
\node (e1) [circle,fill=black,inner sep=1pt] at (6,12) {};
\node (e2) [circle,fill=black,inner sep=1pt] at (12,0) {};
\node (e3) [circle,fill=black,inner sep=1pt] at (0,0) {};
\node (p174) [circle,fill=black,inner sep=1pt] at (7.5,1) {};
\node (p228) [circle,fill=black,inner sep=1pt] at (3,2) {};
\node (p390) [circle,fill=black,inner sep=1pt] at (10.5,3) {};
\node (p444) [circle,fill=black,inner sep=1pt] at (6,4) {};
\node (p660) [circle,fill=black,inner sep=1pt] at (9,6) {};
\node (p714) [circle,fill=black,inner sep=1pt] at (4.5,7) {};
\node (p930) [circle,fill=black,inner sep=1pt] at (7.5,9) {};
\draw [-] (e3) to node[rotate=35, gap, inner sep=1.5pt] {\tiny $x^2\!:\!y^2$} (p228);
\draw [-] (p228) to node[rotate=35, gap, inner sep=1.5pt] {\tiny $x^2\!:\!y^2$} (p444);
\draw [-] (p444) to node[rotate=35, gap, inner sep=1.5pt] {\tiny $x^2\!:\!y^2$} (p660);
\draw [-] (e3) to node[rotate=5, gap, inner sep=1.5pt] {\scriptsize $y\!:\!x^7$} (p174);
\draw [-] (e3) to node[rotate=65, gap, inner sep=1.75pt] {\scriptsize $x\!:\!y^7$} (p714);
\draw [-] (e1) to node[rotate=75, gap, inner sep=1.5pt] {\scriptsize $z\!:\!y^4$} (p714);
\draw [-] (p714) to node[rotate=75, gap, inner sep=1.5pt] {\scriptsize $z\!:\!y^4$} (p228);
\draw [-] (e2) to node[rotate=-12.5, gap, inner sep=2pt] {\scriptsize $z\!:\!x^4$} (p174);
\draw [-] (p174) to node[rotate=-12.5, gap, inner sep=2pt] {\scriptsize $z\!:\!x^4$} (p228);
\draw [-] (p714) to node[rotate=-60, gap, inner sep=1.5pt] {\scriptsize $z^2\!:\!xy$} (p444);
\draw [-] (p444) to node[rotate=-60, gap, inner sep=1.5pt] {\scriptsize $z^2\!:\!xy$} (p174);
\draw [-] (p714) to node[rotate=35, gap, inner sep=1.5pt] {\scriptsize $x\!:\!y^3z$} (p930);
\draw [-] (p174) to node[rotate=35, gap, inner sep=1.5pt] {\scriptsize $y\!:\!x^3z$} (p390);
\draw [-] (p930) to node[rotate=75, gap, inner sep=1.5pt] {\scriptsize $xz^2\!:\!y^3$} (p444);
\draw [-] (p390) to node[rotate=-15, gap, inner sep=1.5pt] {\scriptsize $yz^2\!:\!x^3$} (p444);
\draw [-] (e1) to node[gap, inner sep=1.5pt] {\tiny $z^3$}  (p930) 
		     to node[gap, inner sep=1.5pt] {\tiny $z^3$}  (p660)
		     to node[gap, inner sep=1.5pt] {\tiny $z^3$}  (p390)
		     to node[gap, inner sep=1.5pt] {\tiny $z^3$}  (e2)   
		     to node[gap, inner sep=1.5pt] {\tiny $x^{12}$} (e3) 
		     to node[gap, inner sep=1.5pt] {\tiny $y^{12}$} (e1);
\node at (4.65,4.55) {$U$};
\node at (5.8,6.65) {$V$};
\end{tikzpicture}
};
\node at (6,0) 
{
\begin{tikzpicture}[scale=0.4]
\node (e1) [circle,fill=black,inner sep=1pt] at (6,12) {};
\node (e2) [circle,fill=black,inner sep=1pt] at (12,0) {};
\node (e3) [circle,fill=black,inner sep=1pt] at (0,0) {};
\node (p174) [circle,draw,inner sep=1pt] at (7.5,1) {\scriptsize$11$};
\node (p228) [circle,draw,inner sep=2pt] at (3,2) {\scriptsize$6$};
\node (p390) [circle,fill=black,inner sep=1pt] at (10.5,3) {};
\node (p444) [circle,draw,inner sep=1pt] at (6,4) {\scriptsize$10$};
\node (p660) [circle,fill=black,inner sep=1pt] at (9,6) {};
\node (p714) [circle,draw,inner sep=2pt] at (4.5,7) {\scriptsize$5$};
\node at (4.65,4.55) {$U$};
\node at (5.8,6.65) {$V$};
\node (p930) [circle,fill=black,inner sep=1pt] at (7.5,9) {};
\draw [-] (e3) to node[gap, inner sep=1.5pt] {\tiny $2$} (p228);
\draw [-] (p228) to node[gap, inner sep=1.5pt] {\tiny $2$} (p444);
\draw [-] (p444) to node[gap, inner sep=1.5pt] {\tiny $2$} (p660);
\draw [-] (e3) to node[gap, inner sep=1.5pt] {\scriptsize $7$} (p174);
\draw [-] (e3) to node[gap, inner sep=1.75pt] {\scriptsize $1$} (p714);
\draw [-] (e1) to node[gap, inner sep=1.5pt] {\scriptsize $4$} (p714);
\draw [-] (p714) to node[gap, inner sep=1.5pt] {\scriptsize $4$} (p228);
\draw [-] (e2) to node[gap, inner sep=2pt] {\scriptsize $4$} (p174);
\draw [-] (p174) to node[gap, inner sep=2pt] {\scriptsize $4$} (p228);
\draw [-] (p714) to node[gap, inner sep=1.5pt] {\scriptsize $8$} (p444);
\draw [-] (p444) to node[gap, inner sep=1.5pt] {\scriptsize $8$} (p174);
\draw [-] (p714) to node[gap, inner sep=1.5pt] {\scriptsize $1$} (p930);
\draw [-] (p174) to node[gap, inner sep=1.5pt] {\scriptsize $7$} (p390);
\draw [-] (p930) to node[gap, inner sep=1.5pt] {\scriptsize $9$} (p444);
\draw [-] (p390) to node[gap, inner sep=1.5pt] {\scriptsize $3$} (p444);
\draw [-] (e1) to node[gap, inner sep=1.5pt] {\scriptsize $0$}  (p930) 
		     to node[gap, inner sep=1.5pt] {\scriptsize $0$}  (p660)
		     to node[gap, inner sep=1.5pt] {\scriptsize $0$}  (p390)
		     to node[gap, inner sep=1.5pt] {\scriptsize $0$}  (e2)   
		     to node[gap, inner sep=1.5pt] {\scriptsize $0$} (e3) 
		     to node[gap, inner sep=1.5pt] {\scriptsize $0$} (e1);
\end{tikzpicture}
};
\end{tikzpicture}
\vspace{-0.25cm}
\caption{Toric description of $\Hilb{A}{\C^3}$ (left) and the corresponding marking of Reid's recipe (right) for $A=\frac{1}{12}(1,7,4)$.}
\label{exa:Ab:marking}
\end{center}
\end{figure}

For example, consider the open sets $U$ and $V$ highlighted in Figure \ref{exa:Ab:marking}, with local coordinates $U\cong\C[\frac{z^2}{xy},\frac{x^2}{y^2},\frac{y^4}{z}]$ and $V\cong\C[\frac{xy}{z^2},\frac{xz^2}{y^3},\frac{y^3z}{x}]$. The basis of the tautological bundles $\RR_i$ in both open sets are shown in Table \ref{exa:A12:basisRi}.

\begin{table}[htp]
\begin{center}
\renewcommand{\arraystretch}{1.5}
\begin{tabular}{ccccccccccccc}
 & $\RR_0$ & $\RR_1$ & $\RR_2$ & $\RR_3$ & $\RR_4$ & $\RR_5$ & $\RR_6$ & $\RR_7$ & $\RR_8$ & $\RR_9$ & $\RR_{10}$ & $\RR_{11}$  \\
\hline
$\ZZ_U$ &  $1$ & $x$ &  $y^2$ & ${\color{red} xy^2}$ & $z$ & $xz$ & $y^2z$ & $y$ & ${\color{red} xy}$ & $y^3$ & ${\color{red} xy^3}$ & $yz$
\\
$\ZZ_V$ &  $1$ & $x$ &  $y^2$ & ${\color{red} yz^2}$ & $z$ & $xz$ & $y^2z$ & $y$ & ${\color{red} z^2}$ & $y^3$ & ${\color{red} y^2z^2}$ & $yz$
\\
\hline
\end{tabular}
\end{center}
\caption{Basis of the fibres of each bundle $\RR_i$ on $U$ and $V$.}
\label{exa:A12:basisRi}
\end{table}%

The open sets $U$ and $V$ cover a rational curve $C$ parametrised by the ratio $(xy:z^2)$. Along this curve we have that $G$-igMod$_C(U)=\Span{xy, xy^2, xy^3}$ and $G$-igMod$_C(V)=\Span{z^2, yz^2, y^2z^2}$, which are generated by $xy$ and $z^2$ respectively. Both monomials belong to $S_{8}$ so the curve $C$ is marked with $\rho_8$. If we look at the socles at the origin of both open sets we find that $Soc(U)=\{xz,xy^3,y^2z\}\in S_{5}\oplus S_{10}\oplus S_{6}$ and $Soc(V)=\{xz,y^2z^2,y^3\}\in S_{5}\oplus S_{10}\oplus S_{9}$. We see that the representations $\rho_5$ and $\rho_{10}$ are involved in both open sets, and a similar calculation with the rest of the open sets give the marking for the compact divisors of Figure \ref{exa:Ab:marking} (right).

There are 4 relations in $\Pic Y$:
\begin{align*}
\RR_5 &= \RR_1\otimes\RR_4	&\RR_6 &= \RR_2\otimes\RR_4 \\
\RR_{10} &= \RR_2\otimes\RR_8	&\RR_{11} &= \RR_4\otimes\RR_7 
\end{align*}
which can be verified in the open sets $U$ and $V$ with the data in Table \ref{exa:A12:basisRi}, as well as in the rest of open sets in the covering of $\Hilb{A}{\C^2}$. Following the cookery of Reid's recipe (see \cite[7.3]{C}) we can define the virtual vector bundles:
\begin{align*}
\FF_5 &:= (\RR_1\otimes\RR_4	)\ominus(\RR_5\otimes\OO_Y) &\FF_6 &:= (\RR_2\otimes\RR_4)\ominus(\RR_6\otimes\OO_Y) \\
\FF_{10} &:= (\RR_2\otimes\RR_8)\ominus(\RR_{10}\otimes\OO_Y)	&\FF_{11} &:= (\RR_4\otimes\RR_7)\ominus(\RR_{11}\otimes\OO_Y)
\end{align*}
and it can be checked that $\deg(c_2(\FF_i)|_{E_j})=\delta_{ij}$ where $E_j$ is the compact divisor in $Y$ marked with $\rho_j$.

\end{exa}

\section{A dihedral example in $\SL(3,\C)$}\label{exa:Dn}

Dihedral groups in $\SL(3,\C)$ are finite groups isomorphic to dihedral groups groups of $\GL(2,\C)$. They are contained in the type (B) of Yau-Yu's classification \cite{YY} which are constructed as follows: take $\overline{G}\subset\GL(2,\C)$ a dihedral group and define 
\[
G = \Span{\begin{pmatrix}g&0\\0&\det(g)^{-1}\end{pmatrix}~\bigg|~g\in\overline{G}}\subset\SL(3,\C)
\]

Taking the group $\overline{G}=\Span{\frac{1}{12}(1,7),\left(\!\begin{smallmatrix}\phantom{-}0&1\\-1&0\end{smallmatrix}\right)}\subset\GL(2,\C)$, the group $\mathbb{D}_{5,2}$ in Riemenschneider notation, then 
\[
G = \Span{\alpha=\frac{1}{12}(1,7,4),\beta=\left(\!\begin{smallmatrix}\phantom{-}0&1&0\\-1&0&0\\\phantom{-}0&0&1\end{smallmatrix}\right)}
\]

The cyclic subgroup $A=\Span{\alpha}=\frac{1}{12}(1,7,4)$, which was treated in Example \ref{exa:Ab}, is normal in $G$ of index 2. The group $G/A=\Span{\overline{\beta}}\cong\Z/2\Z$ acts on $\Irr A=\{\rho_0,\ldots,\rho_{11}\}$ by conjugation, such that 
\begin{itemize}
\item[(a)] it fixes $\rho_0$, $\rho_2$, $\rho_4$, $\rho_6$, $\rho_8$ and $\rho_{10}$
\item[(b)] it interchanges $\rho_1$ with $\rho_7$; $\rho_3$ with $\rho_9$; and $\rho_5$ with $\rho_{11}$.
\end{itemize}

The fixed representations in $(a)$ produce 12 irreducible 1-dimensional representations of $G$: $\rho_0^+$ (trivial), $\rho_0^-$, $\rho_2^+$, $\rho_2^-$, $\rho_4^+$, $\rho_4^-$, $\rho_6^+$, $\rho_6^-$, $\rho_8^+$, $\rho_8^-$, $\rho_{10}^+$ and $\rho_{10}^-$. The orbits in $(b)$ produce 3 irreducible 2-dimensional representations of $G$: $V_1$, $V_9$ and $V_5$. In total $G$ has 15 irreducible representations. The McKay quiver $Q$ and relations $R$ for this case is shown in Figure \ref{exa:Dn:McKayQ} where the left and right sides are identified. The quiver can be obtained from the McKay quiver of $\overline{G}$ by adding the horizontal arrows. 

\begin{figure}[htbp]
\begin{center}
\begin{tikzpicture}

\node at (-1,0) 
{
\begin{tikzpicture}[scale=0.5,outer sep=4pt,inner sep=-2pt]
\node (0pizq) at (0,8) {\scriptsize $0+$};
\node (0mizq) at (0,6) {\scriptsize $0-$};
\node (izq) at (-0.75,4) {};
\node (6pizq) at (0,2) {\scriptsize $6+$};
\node (6mizq) at (0,0) {\scriptsize $6-$};
\node (1) at (2,4) {\scriptsize $1$};
\node (8p) at (4,8) {\scriptsize $8+$};
\node (8m) at (4,6) {\scriptsize $8-$};
\node (2p) at (4,2) {\scriptsize $2+$};
\node (2m) at (4,0) {\scriptsize $2-$};
\node (9) at (6,4) {\scriptsize $9$};
\node (4p) at (8,8) {\scriptsize $4+$};
\node (4m) at (8,6) {\scriptsize $4-$};
\node (10p) at (8,2) {\scriptsize $10+$};
\node (10m) at (8,0) {\scriptsize $10-$};
\node (5) at (10,4) {\scriptsize $5$};
\node (0pder) at (12,8) {\scriptsize $0+$};
\node (0mder) at (12,6) {\scriptsize $0-$};
\node (der) at (12.75,4) {};
\node (6pder) at (12,2) {\scriptsize $6+$};
\node (6mder) at (12,0) {\scriptsize $6-$};

\draw [->] (0pder) to node[above, inner sep=-2pt] {\scriptsize $j_2$} (4p);
\draw [->] (0mder) to node[above, inner sep=-2pt] {\scriptsize $k_2$} (4m);
\draw [->] (der) to node[gap, pos=0.25, inner sep=0pt] {\scriptsize $z_0$} (5);
\draw [->] (6pder) to node[below, inner sep=-2pt] {\scriptsize $m_2$} (10p);
\draw [->] (6mder) to node[below, inner sep=-2pt] {\scriptsize $n_2$} (10m);

\draw [->] (4p) to node[above, inner sep=-2pt] {\scriptsize $j_1$} (8p);
\draw [->] (4m) to node[above, inner sep=-2pt] {\scriptsize $k_1$} (8m);
\draw [->] (5) to node[gap, inner sep=0pt] {\scriptsize $z_2$} (9);
\draw [->] (10p) to node[below, inner sep=-2pt] {\scriptsize $m_1$} (2p);
\draw [->] (10m) to node[below, inner sep=-2pt] {\scriptsize $n_1$} (2m);

\draw [->] (8p) to node[above, inner sep=-2pt] {\scriptsize $j_0$} (0pizq);
\draw [->] (8m) to node[above, inner sep=-2pt] {\scriptsize $k_0$} (0mizq);
\draw [->] (9) to node[gap, inner sep=0pt] {\scriptsize $z_1$} (1);
\draw [->] (1) to node[gap, pos=0.65, inner sep=0pt] {\scriptsize $z_0$} (izq);
\draw [->] (2p) to node[below, inner sep=-2pt] {\scriptsize $m_0$} (6pizq);
\draw [->] (2m) to node[below, inner sep=-2pt] {\scriptsize $n_0$} (6mizq);

\draw [->] (0pizq) to node[above, pos=0.45, inner sep=5pt] {\scriptsize $a_0$} (1);
\draw [->] (0mizq) to node[below, very near start, inner sep=0pt] {\scriptsize $c_0$} (1);
\draw [->] (6pizq) to node[above, very near start, inner sep=0pt] {\scriptsize $f_0$} (1);
\draw [->] (6mizq) to node[below, pos=0.44, inner sep=5pt] {\scriptsize $h_0$} (1);

\draw [->] (1) to node[above, pos=0.55, inner sep=5pt] {\scriptsize $b_0$} (8p);
\draw [->] (1) to node[below, very near end, inner sep=0pt] {\scriptsize $d_0$} (8m);
\draw [->] (1) to node[above, very near end, inner sep=0pt] {\scriptsize $e_0$} (2p);
\draw [->] (1) to node[below, pos=0.55, inner sep=5pt] {\scriptsize $g_0$} (2m);

\draw [->] (8p) to node[above, pos=0.45, inner sep=5pt] {\scriptsize $a_1$} (9);
\draw [->] (8m) to node[below, very near start, inner sep=0pt] {\scriptsize $c_1$} (9);
\draw [->] (2p) to node[above, very near start, inner sep=0pt] {\scriptsize $f_1$} (9);
\draw [->] (2m) to node[below, pos=0.45, inner sep=5pt] {\scriptsize $h_1$} (9);

\draw [->] (9) to node[above, pos=0.55, inner sep=5pt] {\scriptsize $b_1$} (4p);
\draw [->] (9) to node[below, very near end, inner sep=0pt] {\scriptsize $d_1$} (4m);
\draw [->] (9) to node[above, very near end, inner sep=0pt] {\scriptsize $e_1$} (10p);
\draw [->] (9) to node[below, pos=0.55, inner sep=5pt] {\scriptsize $g_1$} (10m);

\draw [->] (4p) to node[above, pos=0.45, inner sep=5pt] {\scriptsize $a_2$} (5);
\draw [->] (4m) to node[below, very near start, inner sep=0pt] {\scriptsize $c_2$} (5);
\draw [->] (10p) to node[above, very near start, inner sep=0pt] {\scriptsize $f_2$} (5);
\draw [->] (10m) to node[below, pos=0.45, inner sep=5pt] {\scriptsize $h_2$} (5);

\draw [->] (5) to node[above, pos=0.55, inner sep=5pt] {\scriptsize $b_2$} (0pder);
\draw [->] (5) to node[below, very near end, inner sep=0pt] {\scriptsize $d_2$} (0mder);
\draw [->] (5) to node[above, very near end, inner sep=0pt] {\scriptsize $e_2$} (6pder);
\draw [->] (5) to node[below, pos=0.55, inner sep=5pt] {\scriptsize $g_2$} (6mder);
\end{tikzpicture}
};

\node at (6.1,1.5) {
{\tiny
$\begin{array}{ccc}
b_0j_0-z_0b_2=0  	&a_1z_1-j_0a_0=0 	&d_0k_0-z_0d_2=0 	\\
b_1j_1-z_1b_0=0  	&a_2z_2-j_1a_1=0 	&d_1k_1-z_1d_0=0 	\\
b_2j_2-z_2b_1=0  	&a_0z_0-j_2a_2=0 	&d_2k_2-z_2d_1=0 	\\
\end{array}$}};

\node at (6.1,0.75) {
{\tiny
$\begin{array}{ccc}
f_1z_1-m_0f_0=0	& e_0m_0-z_0e_2=0 	& h_1z_1-n_0h_0=0	\\  
f_2z_2-m_1f_1=0	& e_1m_1-z_1e_0=0 	& h_2z_2-n_1h_1=0	\\  
f_0z_0-m_2f_2=0	& e_2m_2-z_2e_1=0 	& h_0z_0-n_2h_2=0	
\end{array}$}};

\node at (6.1,0) {
{\tiny
$\begin{array}{ccc}
e_0m_0-z_0e_2=0	& g_0n_0-z_0g_2=0 & c_1z_1-k_0c_0=0\\
e_1m_1-z_1e_0=0	& g_1n_1-z_1g_0=0 & c_2z_2-k_1c_1=0\\
e_2m_2-z_2e_1=0	& g_2n_2-z_2g_1=0 & c_0z_0-k_2c_2=0\\
\end{array}$}};

\node at (6.1,-0.75) {
{\tiny
$\begin{array}{ccc}
c_0d_0=0  	& f_0e_0=0 	& h_0g_0=0 	\\
c_1d_1=0  	& f_1e_1=0 	& h_1g_1=0 	 \\
c_2d_2=0  	& f_2e_2=0 	& h_2g_2=0 
\end{array}$}};

\node at (6.1,-1.5) {
{\tiny
$\begin{array}{c}
b_2a_0+d_2c_0-e_2f_0-g_2h_0=0 \\
b_0a_1+d_0c_1-e_0f_1-g_0h_1=0 \\
b_1a_2+d_1c_2-e_1f_2-g_1h_2=0
\end{array}$}};
\end{tikzpicture}
\vspace{-0.75cm}
\caption{McKay quiver $Q$ and relations $R$ for the dihedral group $G$.}
\label{exa:Dn:McKayQ}
\end{center}
\end{figure}
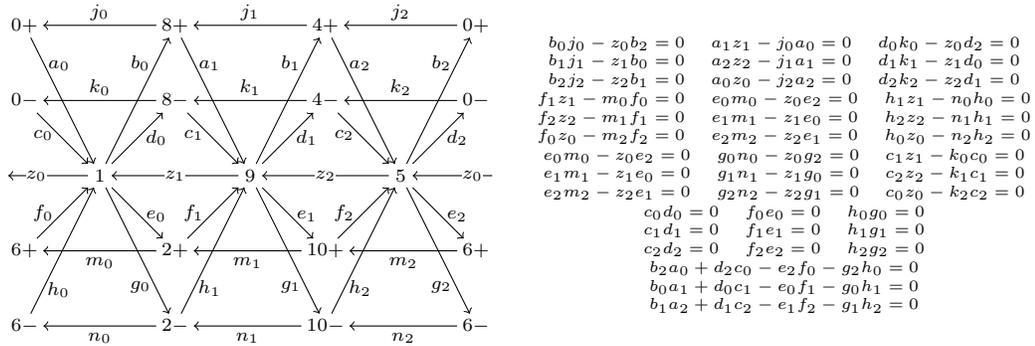

The action of $G/A$ also extends to the coinvariant algebras $S_\rho$ for the group $A$ (see Table \ref{ex:Ab:Srho}). This action forms the coinvariant algebras of $G$, namely $S_i^\pm:=S_{\rho_i^\pm}$ for $i=0,2,4,6,8,10$, and $S_j:=S_{V_j}$ for $j=1,5,9$. These algebras are related by the McKay quiver using the following maps: 
\begin{equation}\label{exa:Dn:xyarrows}
\begin{array}{c}
a = (x,y), b= \begin{pmatrix}y\\-x\end{pmatrix}, c= (x,-y), d=\begin{pmatrix}y\\x\end{pmatrix} \\
e= \begin{pmatrix}x\\y\end{pmatrix}, f= (y,-x), g=\begin{pmatrix}x\\-y\end{pmatrix}, h = (y,x)
\end{array}
\end{equation}

To obtain the elements in each coinvariant algebra we use the maps $(\ref{exa:Dn:xyarrows})$ of the McKay quiver starting from $S_0^+$. For instance, using the same notation as in Example \ref{exa:D4} we have that $ag=x^2+y^2=\plus\in S_2^+$ and $ah=x^2-y^2=\plus\in S_2^-$. Also, the minimal generators of the ring of invariants $\C[x,y,z]^G$ are the polynomials:
{\small
\begin{align*}
z^3 &= jjj				& z\plus^4 &= aefbaefbj		& x^6y^6 &= (adcbadcbadcb)/2^6 \\
z^2\plus^2 &= aefbjj	& x^4y^4z &= (adcbadcbj)/2^4		& x^4y^4\plus^2 &= (adcbadcbaefb)/2^4 \\
x^2y^2z^2 & = (adcbjj)/4		& x^2y^2z\plus^2 &= (adcbaefbj)/4	& \plus^6 &= aefbaefbaefb \\
	xy\plus\minus &= -(aefdcb)/2	& 							& 
\end{align*}
}
which coincide with the nonzero independent loops at the vertex $0+$ of $Q$ (compare with the choice of generators used in \cite[$\S$2.3.1]{YY}). Other elements in some of the rest coinvariant algebras for $G$ are shown in Table \ref{exa:Dn:Coinv} as well as their corresponding maps in $Q$.

{\small
\begin{table}[h]
\begin{center}
\begin{tabular}{|c|l|l||c|l|l|}
\hline
$S_2^+$ 	& $\plus$ & ae 				& $S_2^-$ 	& $\minus$ &  ag \\
			& $2xyz\minus$ & adcem 		& & $2xyz\plus$ & adcgn \\
			& $-2xy\minus^3$ & aghbaghe & & $2xy\plus^3$ & aefbaefg \\
\hline
$S_4^+$ 	& $z$ & j 				& $S_4^-$ 	& $-\plus\minus$ & aefd \\
			& $4x^2y^2$ & adcb 		& & $2xyz^2$ & adkk \\
			& $\plus^2$ & aefb 		& & $2xy\plus^4$ & aefbaefbad \\
			& $-\minus^2$ & aghb 	& & $2xy\minus^4$ & aghbaghbad \\
\hline
$S_1$ 	& $(x,y)$ & a 					& $S_9$ 	& $(2x^2y,-2xy^2)$	 & adc \\
		& $(yz\plus,-xz\plus)$ & aemf 		& & $(y\plus,-x\plus)$	& aef  \\
	 	& $(yz\minus,xz\minus)$ & agnh 	& & $(y\minus,x\minus)$ & agh \\
 		& $(y\plus^3, -x\plus^3)$ &  aefbaef	& &  $(xz^2,yz^2)$ & jja \\
	 	& $(-y\minus^3,-x\minus^3)$ & aghbagh	& & $(x\plus^3\minus,-y\plus^3\minus)$ & aefbaghdc \\
 		& $(y\plus^2\minus, x\plus^2\minus)$ & aefbagh	& & $(x\plus^4,y\plus^4)$ & aefbaefba\\	
\hline
\end{tabular}
\end{center}
\caption{Elements in some of the coinvariant algebras for the group $G$ and their corresponding paths in $Q$,  where $\plus:=x^2+y^2$ and $\minus:=x^2-y^2$.}
\label{exa:Dn:Coinv}
\end{table}%
}

Analogous computations as in Example \ref{exa:D4} give an affine open cover of $Y=\Hilb{\mathbb{D}_{5,2}}{\C^3}$ consisting of 15 open sets, $U_1,\ldots, U_{15}$, which are either copies of $\C^3_{a,b,c}$ or nonsingular hypersurfaces $f\subset\C^4_{a,b,c,d}$. The linear independent paths (or open conditions) for each open set are listed in Table \ref{exa:Dn:OpenConds}, and the equations and local coordinates are shown in Table \ref{exa:Dn:Opens}.

{\small
\begin{table}[h]
\begin{center}
\begin{tabular}{cl}
Open & Linearly independent paths in $Q$ (open conditions) \\
\hline
$U_1$ 	&  $a_0=c_1=a_2=(1,0)$, $h_0=f_1=h_2=(0,1)$, $f_0b_0=f_0d_0=a_0e_0=a_0g_0=1$, \\
		& $f_1b_1=f_1d_1=a_1e_1=a_1g_1=1$, $h_2d_2=a_2e_2=a_2g_2=1$ \\
\hline
$U_2$ 	&  $a_0=a_1=a_2=(1,0)$, $f_0=f_1=f_2=(0,1)$, $j_1=k_1=j_2=1$, \\			& $a_0e_0=a_0g_0=1$, $f_1d_1=a_1e_1=a_1g_1=1$, $f_2d_2=a_1e_2=a_2g_2=1$ \\
\hline
$U_3$ 	&  $a_0=a_1=a_2=(1,0)$, $f_0=f_1=f_2=(0,1)$, $k_0=m_0=n_0=j_1=j_2=1$, \\
		& $f_0d_0=a_0e_0=a_0g_0=1$, $f_1d_1=a_1e_1=a_1g_1=1$\\
\hline
$U_4$ 	& $a_0=c_1=a_2=(1,0)$,	$f_0=f_1=h_2=(0,1)$, $k_0=m_0=n_0=j_1=j_2=1$, \\
		& $a_0e_0=a_0g_0=a_0d_0=1$, $f_1d_1=a_1e_1=a_1g_1=1$ \\
\hline
$U_5$ 	&  $a_0=c_1=a_2=(1,0)$,	$f_0=f_1=h_2=(0,1)$,
	 	    $f_0b_0=a_0d_0=a_0e_0=a_0g_0=1$, \\
		& $f_1b_1=f_1d_1=c_1e_1=c_1g_1=1$, $h_2d_2=a_2e_2=a_2g_2=1$ \\
\hline
$U_6$ 	&  $a_0=a_1=a_2=(1,0)$,	$f_0=f_1=f_2=(0,1)$,
	 	    $f_0b_0=a_0d_0=a_0e_0=f_0g_0=1$, \\
		& $f_1b_1=a_1d_1=a_1e_1=f_1g_1=1$, $f_2d_2=a_2e_2=f_2g_2=1$ \\
\hline
$U_7$ 	&  $a_0=a_1=a_2=(1,0)$, $f_0=f_1=f_2=(0,1)$, $j_1=n_1=j_2=1$ \\
		& $a_0d_0=a_0e_0=f_0g_0=1$, $a_1d_1=a_1e_1=f_1g_1=1$, $a_2d_2=a_2e_2=f_2g_2=1$\\
\hline
$U_8$ 	&  $a_0=a_1=a_2=(1,0)$, $f_0=f_1=f_2=(0,1)$, $j_1=j_2=1$, \\
		& $a_0d_0=a_0e_0=a_0g_0=1$, $a_1d_1=a_1e_1=f_1g_1=1$, $a_2d_2=a_2e_2=a_2g_2=1$\\
\hline
$U_9$ 	&  $a_0=a_1=a_2=(1,0)$, $f_0=f_1=f_2=(0,1)$, $k_0=m_0=n_0=j_1=j_2=1$, \\
		& $a_0d_0=a_0e_0=a_0g_0=1$, $a_1d_1=a_1e_1=a_1g_1=1$\\
\hline
$U_{10}$ 	& $a_0=c_1=a_2=(1,0)$, $h_0=h_1=f_2=(0,1)$, $h_0b_0=a_0d_0=a_0e_0=a_0g_0=1$\\
		&  $h_1b_1=a_1d_1=a_1e_1=a_1g_1=1$,  $f_2d_2=a_2e_2=a_2g_2=1$\\
\hline
$U_{11}$ 	& $a_0=a_1=a_2=(1,0)$, $h_0=h_1=h_2=(0,1)$, $h_0b_0=a_0d_0=h_0e_0=a_0g_0=1$\\
		&  $h_1b_1=a_1d_1=h_1e_1=a_1g_1=1$,  $a_2d_2=h_2e_2=a_2g_2=1$\\
\hline
$U_{12}$ 	& $a_0=a_1=a_2=(1,0)$, $h_0=h_1=h_2=(0,1)$, $a_0d_0=a_0g_0=1$, \\
		& $h_1e_1=1$, $k_0=m_0=n_0=j_1=m_1=j_2=k_2=n_2=1$\\
\hline
$U_{13}$ 	&  $a_0=a_1=a_2=(1,0)$, $h_0=h_1=h_2=(0,1)$, $k_0=n_0=j_1=j_2=k_2=n_2=1$, \\
		& $a_0d_0=a_0e_0=a_0g_0=1$, $h_1e_1=1$, $a_2e_2=1$ \\
\hline
$U_{14}$ 	&  $a_0=c_1=a_2=(1,0)$, $h_0=h_1=f_2=(0,1)$, $k_0=j_1=j_2=1$, \\
		& $a_0d_0=a_0e_0=a_0g_0=1$, $h_1d_1=a_1e_1=a_1g_1=1$, $a_2e_2=a_2g_2=1$ \\
\hline
$U_{15}$ 	&  $a_0=a_1=a_2=(1,0)$, $h_0=h_1=h_2=(0,1)$, $k_0=j_1=j_2=1$, \\
		& $a_0d_0=a_0e_0=a_0g_0=1$, $a_1d_1=a_1e_1=a_1g_1=1$, $a_2e_2=a_2g_2=1$ \\
\hline
\end{tabular}
\end{center}
\caption{Open conditions for each open set $U_i\subset\Hilb{\mathbb{D}_{5,2}}{\C^3}$,  for $i=1,\ldots,15$.}
\label{exa:Dn:OpenConds}
\end{table}%
}

{\scriptsize
\begin{table}[h]
\begin{center}
\begin{tabular}{ccccccc}
Open & $f\subset\C^4_{a,b,c,d}$ & a & b & c & d & Socle at $(0,0,0)$ \\
\hline
$U_1$ 		&  $d(a^2c+1)=ac$ & $\frac{2xy}{\plus^3\minus}$ & $\frac{z}{\plus^2}$ & $\plus^6$ & $2xy\plus\minus$  & $\rho^-_0$ \\
$U_2$ 		&  $d(1-a^2bc)=ab$ & $\frac{2xy}{z\plus\minus}$ & $\frac{\minus^2}{z}$ & $z^3$ & $\frac{2xy\plus}{z^2\minus}$ &  $\rho^-_0$ \\
$U_3$ 		& $d(ac+1)=1$ & $\frac{2xyz^2}{\plus\minus}$ & $\frac{z\plus\minus}{2xy}$ & $\frac{2xy\plus}{z^2\minus}$ & $\frac{\minus^2}{\plus^2}$ & $\rho^-_0$, $\rho^-_4$, $V_5$ \\
$U_4$ 		& $\C^3_{a,b,c}$ & $\frac{\plus^2}{z}$ & $\frac{z^2\plus}{2xy\minus}$ & $\frac{\minus^2}{\plus^2}$ & & $\rho^-_6$, $\rho^+_8$, $V_1$, $V_5$\\
$U_5$ 		& $d(1-c)=ac$ & $\frac{2xy\plus^3}{\minus}$ & $\frac{z}{\plus^2}$ & $\frac{\minus^2}{\plus^2}$ & $\frac{\plus^3\minus}{2xy}$ & $\rho^-_6$, $\rho^+_8$ \\
$U_6$ 		& $d(ab^2+1)=1$ & $\plus^6$ & $\frac{\minus}{2xy\plus^3}$ & $\frac{z}{\plus^2}$ & $\frac{4x^2y^2}{\plus^2}$ & $\rho^-_6$\\
$U_7$ 		& $d(ab^2c+1)=1$ & $\frac{\plus^2}{z}$ & $\frac{\minus}{2xyz\plus}$ & $z^3$ & $\frac{4x^2y^2}{\plus^2}$ & $\rho^-_6$ \\
$U_8$ 		& $d(ab+1)=ab$ & $\frac{\plus\minus}{2xyz^2}$ & $\frac{z^2\minus}{2xy\plus}$ & $\frac{2xyz\plus}{\minus}$ & $\frac{\minus^2}{\plus^2}$ & $\rho^-_6$, $\rho^-_{10}$, $V_5$ \\
$U_9$ 		& $d(1-c)=ac$ & $\frac{\plus\minus}{2xyz^2}$ & $z^3$ & $\frac{4x^2y^2}{\plus^2}$ & $\frac{2xy\plus}{z^2\minus}$ & $\rho^-_{10}$, $V_5$ \\
$U_{10}$ 	& $d(c+1)=ac$ & $\frac{\plus\minus^3}{2xy}$ & $\frac{z}{\minus^2}$ & $\frac{4x^2y^2}{\minus^2}$ & $\frac{2xy\minus^3}{\plus}$ & $\rho^+_6$, $\rho^+_8$ \\
$U_{11}$ 	& $d(ab^2-1)=1$ & $\minus^6$ & $\frac{\plus}{2xy\minus^3}$ & $\frac{z}{\minus^2}$ & $\frac{4x^2y^2}{\minus^2}$ & $\rho^+_6$\\
$U_{12}$ 	& $d(ab^2c-1)=1$ & $\frac{\minus^2}{z}$ & $\frac{\plus}{2xyz\minus}$ & $z^3$ & $\frac{4x^2y^2}{\minus^2}$ & $\rho^+_6$\\
$U_{13}$ 	& $d(ab-1)=1$ & $\frac{\plus\minus}{2xyz^2}$ & $\frac{z^2\plus}{2xy\minus}$ & $\frac{2xyz\minus}{\plus}$ & $\frac{4x^2y^2}{\minus^2}$ & $\rho^+_6$, $\rho^+_{10}$, $V_5$ \\
$U_{14}$ 	& $\C^3_{a,b,c}$ & $\frac{\minus^2}{z}$ & $\frac{z^2\minus}{2xy\plus}$ & $\frac{\plus^2}{\minus^2}$ & & $\rho^+_6$, $\rho^+_8$, $V_1$, $V_5$ \\
$U_{15}$ 	& $d(c+1)=ac$ & $\frac{\plus\minus}{2xyz^2}$ & $b=z^3$ & $\frac{4x^2y^2}{\minus^2}$ & $\frac{2xy\minus}{z^2\plus}$ & $\rho^+_{10}$, $V_5$ \\
\hline
\end{tabular}
\end{center}
\caption{Open cover of $\Hilb{G}{\C^3}=\bigcup_{i=1}^{15}U_i$ with their local coordinates and socles at the origin of every $U_i$.}
\label{exa:Dn:Opens}
\end{table}%
}

The configuration of the exceptional locus in $Y$ is shown in Figure \ref{exa:Dn:bento} (left), which resembles the shape of a {\em bento box}. It consists of five irreducible compact divisors $E_1,\ldots, E_5$, which agrees with the number of conjugacy classes of age 2 as expected (see \cite[1.6]{IR}).

\begin{figure}[htbp]
\begin{center}
\begin{tikzpicture}

\node at (0,0) {
\begin{tikzpicture}[scale=0.4,inner sep = 0pt]
\node (u1) at (4,3) {\scriptsize $U_1$};
\node (u2) at (5.5,4.5) {\scriptsize $U_2$};
\node (u3) at (7,4) {\scriptsize $U_3$};
\node (u4) at (3,2) {\scriptsize $U_4$};
\node (u5) at (0,0) {\scriptsize $U_5$};
\node (u6) at (0,3) {\scriptsize $U_6$};
\node (u7) at (1.5,4.5) {\scriptsize $U_7$};
\node (u8) at (3,4) {\scriptsize $U_8$};
\node (u9) at (6,6) {\scriptsize $U_9$};
\node (u10) at (8,0) {\scriptsize $U_{10}$};
\node (u11) at (8,3) {\scriptsize $U_{11}$};
\node (u12) at (9.5,4.5) {\scriptsize $U_{12}$};
\node (u13) at (11,4) {\scriptsize $U_{13}$};
\node (u14) at (11,2) {\scriptsize $U_{14}$};
\node (u15) at (14,6) {\scriptsize $U_{15}$};

\node (E2) at (1.5,2.5) {\color{red}{$E_2$}};
\node (E4) at (5.5,2.5) {\color{red}{$E_4$}};
\node (E3) at (9.5,2.5) {\color{red}{$E_3$}};
\node (E1) at (7,1) {\color{red}{$E_1$}};
\node (E5) at (6.75,5) {\color{red}{$E_5$}};

\draw (4,2) -- (u4) -- (u5) -- (u10) -- (u14);
\draw (7,2) -- (8,2);
\draw (4,0) -- (u1) -- (u2) -- (u3) -- (7,2);
\draw (u10) -- (u11) -- (u12) -- (u13) -- (u14);
\draw (7,2) -- (4,0);
\draw (u5) -- (u6) -- (u7) -- (u8) -- (u4);
\draw (u8) -- (u9) -- (u15) -- (u13);
\draw (u3) -- (10,6);
\end{tikzpicture}
};

\node at (5.5,0) {
\begin{tikzpicture}[scale=0.4,inner sep = 0pt]
\coordinate (u1) at (4,3);
\coordinate (u2) at (5.5,4.5);
\coordinate (u3) at (7,4);
\coordinate (u4) at (3,2);
\coordinate (u5) at (0,0);
\coordinate (u6) at (0,3);
\coordinate (u7) at (1.5,4.5);
\coordinate (u8) at (3,4);
\coordinate (u9) at (6,6);
\coordinate (u10) at (8,0);
\coordinate (u11) at (8,3);
\coordinate (u12) at (9.5,4.5);
\coordinate (u13) at (11,4);
\coordinate (u14) at (11,2);
\coordinate (u15) at (14,6);
\coordinate [label=center:{\color{red}{\small$6-$}}] (E2) at (1.5,2.5);
\coordinate [label=center:{\color{red}{\small$0-$}}] (E4) at (5.5,2.5);
\coordinate [label=center:{\color{red}{\small$6+$}}] (E3) at (9.5,2.5);
\coordinate [label=center:{\color{red}{\small$8+$}}] (E1) at (7,1);
\coordinate [label=center:{\color{red}{\small$5$}}](E5) at (6.75,5);

\node (e1) at (2,0) [gap,inner sep=1pt] {\tiny$4+$};
\node (e2) at (6,0) [gap,inner sep=1pt] {\tiny$4+$};
\node (e3) at (1.5,1) [gap,inner sep=1pt] {\tiny$4+$};
\node (e4) at (5.5,1) [gap,inner sep=1pt] {\tiny$4+$};
\node (e5) at (9.5,1) [gap,inner sep=1pt] {\tiny$4+$};
\node (e6) at (0,1.25) [gap,inner sep=1pt] {\tiny$2-$};
\node (e7) at (4,1.25) [gap,inner sep=1pt] {\tiny$8-$};
\node (e8) at (8,1.25) [gap,inner sep=1pt] {\tiny$2+$};
\node (e9) at (0.75,3.75) [gap,inner sep=1pt] {\tiny$4+$};
\node (e10) at (4.75,3.75) [gap,inner sep=1pt] {\tiny$4+$};
\node (e11) at (8.75,3.75) [gap,inner sep=1pt] {\tiny$4+$};
\node (e12) at (3.5,2) [gap,inner sep=0pt] {\tiny$1$};
\node (e13) at (7.5,2) [gap,inner sep=0pt] {\tiny$1$};
\node (e14) at (3,3) [gap,inner sep=2pt] {\tiny$9$};
\node (e15) at (7,3) [gap,inner sep=2pt] {\tiny$9$};
\node (e16) at (11,3) [gap,inner sep=2pt] {\tiny$9$};
\node (e17) at (2.25,4.25) [gap,inner sep=1pt] {\tiny$2-$};
\node (e18) at (6.25,4.25) [gap,inner sep=1pt] {\tiny$8-$};
\node (e19) at (10.25,4.25) [gap,inner sep=1pt] {\tiny$2+$};
\node (e20) at (4.5,5) [gap,inner sep=0.7pt] {\tiny$10-$};
\node (e21) at (8.5,5) [gap,inner sep=1pt] {\tiny$4-$};
\node (e22) at (12.5,5) [gap,inner sep=0.7pt] {\tiny$10+$};
\node (e23) at (8,6) [gap,inner sep=2pt] {\tiny$9$};
\node (e24) at (12,6) [gap,inner sep=2pt] {\tiny$9$};
\draw (u5) -- (e1) -- (e2) -- (u10) -- (e5) -- (u14) -- (e16) --
   	(u13) -- (e22) -- (u15) -- (e24) -- (e23) -- (u9) -- (e20) --
	(u8) -- (e17) -- (u7) -- (e9) -- (u6) -- (e6) -- (u5);
\draw (4,0) -- (e7) -- (u1) -- (e10) -- (u2) -- (e18) -- (u3) -- 
	(e15) -- (7,2) -- (e4) -- (4,0);
\draw (u10) -- (e8) -- (u11) -- (e11) -- (u12) -- (e19) -- (u13);
\draw (u5) -- (e3) -- (u4) -- (e14) -- (u8);
\draw (u4) -- (e12) -- (4,2);
\draw (7,2) -- (e13) -- (8,2);
\draw (u3) -- (e21) -- (10,6);
\end{tikzpicture}
};
\end{tikzpicture}
\vspace{-0.25cm}
\caption{Exceptional locus in $\Hilb{G}{\C^3}$ (left) and the marking of each irreducible component (right).}
\label{exa:Dn:bento}
\end{center}
\end{figure}
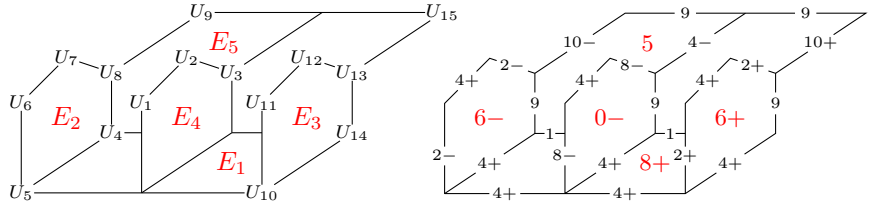

The marking of the irreducible components is shown in Figure \ref{exa:Dn:bento} (right). For the marking of compact divisors the irreducible representations containing elements in the socles at the origin of every open set are shown in Table \ref{exa:Dn:Opens}. The socles are detected from the representation space of $Q$ at an open set $U$ by setting all local parameters equal to zero, and looking at the vertices (representations) where every arrow going out is the zero map. For example, for $U_8$ we see that every arrow out of $\rho_6^-$ and $\rho_{10}^-$, as well as for the basis element $e_2\in\C^2_5$. This means that $Soc(U_8)=\{z\minus,2xy\plus,(yz^\plus,-xz^2\plus)\}$. The calculation with the rest of the open sets give us that $E_1$ is marked with $\rho_8^+$, $E_2$ with $\rho_6^-$, $E_3$ with $\rho_6^+$, $E_4$ with $\rho_0^-$ and $E_5$ with $V_5$. 

For the marking of rational curves $C\subset Y$ the basis of sections of the fibres of $\RR$ along the open covering of $Y$ are shown in Table \ref{exa:Dn:basisRi}. From them one can compute the $G$-igsaw transformation modules $G$-igMod$_{C}$. To obtain the elements of the table, let us say that for $U_i$, $i\neq1$, every basis element is obtained directly using the open conditions in Table \ref{exa:Dn:OpenConds}. For the open set $U_1$ we have a relation in $S_4^+$ of the form 
\[\minus^2=-(ad-1)\plus^2\]
Since the equation of $U_1$ can be written as $ac(ad-1)=-d$ se can say that $ad-1\neq0$ on $U_1$, which implies that we can choose either $\plus^2$ or $\minus^2$ as a basis element in $\RR$.

Let us explain the calculations in detail for the open sets $U_7$ and $U_8$. The choices of nonzero arrows in order to generate the whole module from $S_0^+$ for both open sets are shown in Figure \ref{exa:Dn:U7U8}. Following these arrows, it can be checked that starting with $1\in S^+_0$ and using the maps in (\ref{exa:Dn:xyarrows}) we produce the basis elements in each $\RR_i$ that appear in Table \ref{exa:Dn:basisRi}.

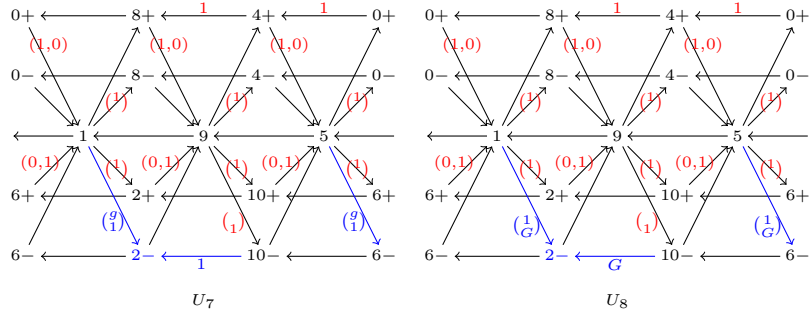
\begin{figure}[htbp]
\begin{center}
\begin{tikzpicture}

\node (n1) at (1,0) {
\begin{tikzpicture}[scale=0.4,outer sep=1pt,inner sep=1pt]
\node (0pizq) at (0,8) {\tiny $0+$};
\node (0mizq) at (0,6) {\tiny $0-$};
\node (izq) at (-0.5,4) {};
\node (6pizq) at (0,2) {\tiny $6+$};
\node (6mizq) at (0,0) {\tiny $6-$};
\node (1) at (2,4) {\tiny $1$};
\node (8p) at (4,8) {\tiny $8+$};
\node (8m) at (4,6) {\tiny $8-$};
\node (2p) at (4,2) {\tiny $2+$};
\node (2m) at (4,0) {\color{blue}\tiny $2-$};
\node (9) at (6,4) {\tiny $9$};
\node (4p) at (8,8) {\tiny $4+$};
\node (4m) at (8,6) {\tiny $4-$};
\node (10p) at (8,2) {\tiny $10+$};
\node (10m) at (8,0) {\tiny $10-$};
\node (5) at (10,4) {\tiny $5$};
\node (0pder) at (12,8) {\tiny $0+$};
\node (0mder) at (12,6) {\tiny $0-$};
\node (der) at (12.5,4) {};
\node (6pder) at (12,2) {\tiny $6+$};
\node (6mder) at (12,0) {\tiny $6-$};
\node (u7) at (6,-1.5) {\tiny $U_7$};

\draw [->] (0pder) to node[above, inner sep=0pt] {\color{red}\tiny $1$} (4p);
\draw [->] (0mder) to node[above, inner sep=0pt] {} (4m);
\draw [->] (der) to node[pos=0.2, inner sep=-2pt] {} (5);
\draw [->] (6pder) to node[below, inner sep=-2pt] {} (10p);
\draw [->] (6mder) to node[below, inner sep=-2pt] {} (10m);

\draw [->] (4p) to node[above, inner sep=0pt] {\color{red}\tiny $1$} (8p);
\draw [->] (4m) to node[above, inner sep=-2pt] {} (8m);
\draw [->] (5) to node[inner sep=-2pt] {} (9);
\draw [->] (10p) to node[below, inner sep=-2pt] {} (2p);
\draw [->,blue] (10m) to node[below, inner sep=0pt] {\tiny $1$} (2m);

\draw [->] (8p) to node[above, inner sep=0pt] {} (0pizq);
\draw [->] (8m) to node[above, inner sep=0pt] {} (0mizq);
\draw [->] (9) to node[inner sep=-2pt] {} (1);
\draw [->] (1) to node[pos=0.8, inner sep=-2pt] {} (izq);
\draw [->] (2p) to node[below, inner sep=0pt] {} (6pizq);
\draw [->] (2m) to node[below, inner sep=0pt] {} (6mizq);

\draw [->] (0pizq) to node[above, pos=0.4, inner sep=4pt] {\color{red}\tiny $(1,\!0)$} (1);
\draw [->] (0mizq) to node[below, pos=0.4, inner sep=0pt] {} (1);
\draw [->] (6pizq) to node[above,pos=0.15, inner sep=2pt] {\color{red}\tiny $(0,\!1)$} (1);
\draw [->] (6mizq) to node[below, pos=0.4, inner sep=5pt] {} (1);

\draw [->] (1) to node[above, inner sep=5pt] {} (8p);
\draw [->] (1) to node[pos=0.6, inner sep=0pt] {\color{red}\tiny $\left(\!\begin{smallmatrix}1\\\phantom{0}\end{smallmatrix}\!\right)$} (8m);
\draw [->] (1) to node[pos=0.6, inner sep=0pt] {\color{red}\tiny $\left(\!\begin{smallmatrix}1\\\phantom{0}\end{smallmatrix}\!\right)$} (2p);
\draw [->,blue] (1) to node[below, inner sep=5pt] {\tiny $\left(\!\begin{smallmatrix}g\\1\end{smallmatrix}\!\right)$} (2m);

\draw [->] (8p) to node[above, pos=0.4, inner sep=4pt] {\color{red}\tiny $(1,\!0)$} (9);
\draw [->] (8m) to node[below, pos=0.4, inner sep=0pt] {} (9);
\draw [->] (2p) to node[above,pos=0.15, inner sep=2pt] {\color{red}\tiny $(0,\!1)$} (9);
\draw [->] (2m) to node[below, pos=0.4, inner sep=5pt] {} (9);

\draw [->] (9) to node[above, inner sep=5pt] {} (4p);
\draw [->] (9) to node[pos=0.6, inner sep=0pt] {\color{red}\tiny $\left(\!\begin{smallmatrix}1\\\phantom{0}\end{smallmatrix}\!\right)$} (4m);
\draw [->] (9) to node[pos=0.6, inner sep=0pt] {\color{red}\tiny $\left(\!\begin{smallmatrix}1\\\phantom{0}\end{smallmatrix}\!\right)$} (10p);
\draw [->] (9) to node[below, inner sep=5pt] {\color{red}\tiny $\left(\!\begin{smallmatrix}\phantom{0}\\1\end{smallmatrix}\!\right)$} (10m);

\draw [->] (4p) to node[above, pos=0.4, inner sep=4pt] {\color{red}\tiny $(1,\!0)$} (5);
\draw [->] (4m) to node[below, pos=0.4, inner sep=0pt] {} (5);
\draw [->] (10p) to node[above,pos=0.15, inner sep=2pt] {\color{red}\tiny $(0,\!1)$} (5);
\draw [->] (10m) to node[below, pos=0.4, inner sep=5pt] {} (5);

\draw [->] (5) to node[above, inner sep=4pt] {} (0pder);
\draw [->] (5) to node[pos=0.6, inner sep=0pt] {\color{red}\tiny $\left(\!\begin{smallmatrix}1\\\phantom{0}\end{smallmatrix}\!\right)$} (0mder);
\draw [->] (5) to node[pos=0.6, inner sep=0pt] {\color{red}\tiny $\left(\!\begin{smallmatrix}1\\\phantom{0}\end{smallmatrix}\!\right)$} (6pder);
\draw [->,blue] (5) to node[below, inner sep=5pt] {\tiny $\left(\!\begin{smallmatrix}g\\1\end{smallmatrix}\!\right)$} (6mder);
\end{tikzpicture}

};

\node (n1) at (6.5,0) {
\begin{tikzpicture}[scale=0.4,outer sep=1pt,inner sep=1pt]
\node (0pizq) at (0,8) {\tiny $0+$};
\node (0mizq) at (0,6) {\tiny $0-$};
\node (izq) at (-0.5,4) {};
\node (6pizq) at (0,2) {\tiny $6+$};
\node (6mizq) at (0,0) {\tiny $6-$};
\node (1) at (2,4) {\tiny $1$};
\node (8p) at (4,8) {\tiny $8+$};
\node (8m) at (4,6) {\tiny $8-$};
\node (2p) at (4,2) {\tiny $2+$};
\node (2m) at (4,0) {\color{blue}\tiny $2-$};
\node (9) at (6,4) {\tiny $9$};
\node (4p) at (8,8) {\tiny $4+$};
\node (4m) at (8,6) {\tiny $4-$};
\node (10p) at (8,2) {\tiny $10+$};
\node (10m) at (8,0) {\tiny $10-$};
\node (5) at (10,4) {\tiny $5$};
\node (0pder) at (12,8) {\tiny $0+$};
\node (0mder) at (12,6) {\tiny $0-$};
\node (der) at (12.5,4) {};
\node (6pder) at (12,2) {\tiny $6+$};
\node (6mder) at (12,0) {\tiny $6-$};
\node (u8) at (6,-1.5) {\tiny $U_8$};

\draw [->] (0pder) to node[above, inner sep=0pt] {\color{red}\tiny $1$} (4p);
\draw [->] (0mder) to node[above, inner sep=0pt] {} (4m);
\draw [->] (der) to node[pos=0.2, inner sep=-2pt] {} (5);
\draw [->] (6pder) to node[below, inner sep=-2pt] {} (10p);
\draw [->] (6mder) to node[below, inner sep=-2pt] {} (10m);

\draw [->] (4p) to node[above, inner sep=0pt] {\color{red}\tiny $1$} (8p);
\draw [->] (4m) to node[above, inner sep=-2pt] {} (8m);
\draw [->] (5) to node[inner sep=-2pt] {} (9);
\draw [->] (10p) to node[below, inner sep=-2pt] {} (2p);
\draw [->,blue] (10m) to node[below, inner sep=0pt] {\tiny $G$} (2m);

\draw [->] (8p) to node[above, inner sep=0pt] {} (0pizq);
\draw [->] (8m) to node[above, inner sep=0pt] {} (0mizq);
\draw [->] (9) to node[inner sep=-2pt] {} (1);
\draw [->] (1) to node[pos=0.8, inner sep=-2pt] {} (izq);
\draw [->] (2p) to node[below, inner sep=0pt] {} (6pizq);
\draw [->] (2m) to node[below, inner sep=0pt] {} (6mizq);

\draw [->] (0pizq) to node[above, pos=0.4, inner sep=4pt] {\color{red}\tiny $(1,\!0)$} (1);
\draw [->] (0mizq) to node[below, pos=0.4, inner sep=0pt] {} (1);
\draw [->] (6pizq) to node[above,pos=0.15, inner sep=2pt] {\color{red}\tiny $(0,\!1)$} (1);
\draw [->] (6mizq) to node[below, pos=0.4, inner sep=5pt] {} (1);

\draw [->] (1) to node[above, inner sep=5pt] {} (8p);
\draw [->] (1) to node[pos=0.6, inner sep=0pt] {\color{red}\tiny $\left(\!\begin{smallmatrix}1\\\phantom{0}\end{smallmatrix}\!\right)$} (8m);
\draw [->] (1) to node[pos=0.6, inner sep=0pt] {\color{red}\tiny $\left(\!\begin{smallmatrix}1\\\phantom{0}\end{smallmatrix}\!\right)$} (2p);
\draw [->,blue] (1) to node[below, inner sep=6.5pt] {\tiny $\left(\!\begin{smallmatrix}1\\G\end{smallmatrix}\!\right)$} (2m);

\draw [->] (8p) to node[above, pos=0.4, inner sep=4pt] {\color{red}\tiny $(1,\!0)$} (9);
\draw [->] (8m) to node[below, pos=0.4, inner sep=0pt] {} (9);
\draw [->] (2p) to node[above,pos=0.15, inner sep=2pt] {\color{red}\tiny $(0,\!1)$} (9);
\draw [->] (2m) to node[below, pos=0.4, inner sep=5pt] {} (9);

\draw [->] (9) to node[above, inner sep=5pt] {} (4p);
\draw [->] (9) to node[pos=0.6, inner sep=0pt] {\color{red}\tiny $\left(\!\begin{smallmatrix}1\\\phantom{0}\end{smallmatrix}\!\right)$} (4m);
\draw [->] (9) to node[pos=0.6, inner sep=0pt] {\color{red}\tiny $\left(\!\begin{smallmatrix}1\\\phantom{0}\end{smallmatrix}\!\right)$} (10p);
\draw [->] (9) to node[below, inner sep=5pt] {\color{red}\tiny $\left(\!\begin{smallmatrix}\phantom{0}\\1\end{smallmatrix}\!\right)$} (10m);

\draw [->] (4p) to node[above, pos=0.4, inner sep=4pt] {\color{red}\tiny $(1,\!0)$} (5);
\draw [->] (4m) to node[below, pos=0.4, inner sep=0pt] {} (5);
\draw [->] (10p) to node[above,pos=0.15, inner sep=2pt] {\color{red}\tiny $(0,\!1)$} (5);
\draw [->] (10m) to node[below, pos=0.4, inner sep=5pt] {} (5);

\draw [->] (5) to node[above, inner sep=4pt] {} (0pder);
\draw [->] (5) to node[pos=0.6, inner sep=0pt] {\color{red}\tiny $\left(\!\begin{smallmatrix}1\\\phantom{0}\end{smallmatrix}\!\right)$} (0mder);
\draw [->] (5) to node[pos=0.6, inner sep=0pt] {\color{red}\tiny $\left(\!\begin{smallmatrix}1\\\phantom{0}\end{smallmatrix}\!\right)$} (6pder);
\draw [->,blue] (5) to node[below, inner sep=6.5pt] {\tiny $\left(\!\begin{smallmatrix}1\\G\end{smallmatrix}\!\right)$} (6mder);
\end{tikzpicture}
};
\end{tikzpicture}
\vspace{-0.5cm}
\caption{Choices of nonzero arrows for $U_7$ and $U_8$. It is highlighted in blue the difference in both sets, where $g$ and $G$ are the (inverse) coordinates of the curve $\LL\cong\PP^1$ covered by $U_7\cup U_8$.}
\label{exa:Dn:U7U8}
\end{center}
\end{figure}

By setting the choices of Figure \ref{exa:Dn:U7U8} (left) and using the relations of the quiver we obtain $U_7$ as the hypersurface $C(Bg^2j+1)=1$ in $\C^4$. This equation tells us that $C$ can be written as a well defined function when $B=g=j=0$ at the origin of the open set, so we can take $B$, $g$ and $j$ as the local coordinates of the affine open set $U_7$. Using again (\ref{exa:Dn:xyarrows}) we see that
\[
B = \frac{\plus^2}{z}, ~~~g = \frac{\minus}{2xyz\plus}, ~~~j=z^3, ~~~C=\frac{4x^2y^2}{\plus^2}
\]
If we restrict to the curve $\LL$ parametrized by the local coordinate $g$, that is, if we set $B=j=0$ (thus $C=1$), we see that $G$-igMod$_\LL(U_7)=\Span{xyz\plus,xyz^2\plus}$. Similarly, we obtain the open set $U_8$ as the hypersurface $H(Dn+1)=Dn$ where
\[
D = \frac{\plus\minus}{2xyz^2}, ~~~G = \frac{2xyz\plus}{\minus}, ~~~n=\frac{z^2\minus}{2xy\plus}, ~~~H=\frac{\minus^2}{\plus^2}
\]
In this open set the curve $\LL$ is parametrised by the local coordinate $G$, and we have that $G$-igMod$_\LL(U_8)=\Span{\minus,z\minus}$. The generators of these submodules are $xyz\plus$ and $\minus$ respectively, both in $S_2^-$, so the curve $\LL$ is marked with the representation $\rho_2^-$. 

Also, at this curve the quiver gives only two nonzero relations among basis elements in $S^-_2$ and $S^-_6$:
\begin{align*}
&\text{at $U_7|_{\LL}$} & &\text{at $U_8|_{\LL}$} &\\
\minus &= g\cdot2xyz\plus &2xyz\plus&=G\cdot\minus &\text{(in $S^-_2$)}\\
z\minus &= g\cdot2xyz^2\plus &2xyz^2\plus&=G\cdot z\minus &\text{(in $S^-_6$)}
\end{align*}
In both open sets the relation in $S^-_6$ is generated by the relation in $S^-_2$, having 
\[
G\cdot\minus = g\cdot2xyz\plus
\]
so we can conclude that $\LL$ is parametrised by the ratio $(g:G)=(\minus:2xyz\plus)$.  \\

{\scriptsize
\begin{table}[htp]
\begin{center}
\setlength{\tabcolsep}{1pt}
\renewcommand{\arraystretch}{1.5}
\begin{tabular}{cccccccccccccccc}
 & $\RR^+_0$ & $\RR^-_0$ & $\RR^+_2$ & $\RR^-_2$ & $\RR^+_4$ & $\RR^-_4$ & $\RR^+_6$ & $\RR^-_6$ & $\RR^+_8$ & $\RR^-_8$ & $\RR^+_{10}$ & $\RR^-_{10}$ & $\det\RR_1$ & $\det\RR_5$ & $\det\RR_9$ \\
\hline
$\ZZ_1$ &  $1$ & $\plus^5\minus$ &  $\plus$ & $\minus$ & $\plus^2$ & $\plus\minus $& $\plus^3$ & $\plus^2\minus$ & $\plus^4$ & $\plus^3\minus$ & $\plus^5$ & $\plus^4\minus$ & $\plus^4$ & $\plus^8$ & $\plus^6$
\\
$\ZZ_2$ &  $1$ & $z^2\plus\minus$  & $\plus$ & $\minus$ & $z$ & $\plus\minus$ & $z\plus$ & $z\minus$ & $z^2$ & $z\plus\minus$ & $z^2\plus$ & $z^2\minus$ & $z\plus^2$ & $z^3\plus^2$ & $z^2\plus^2$
\\
$\ZZ_3$ &  $1$ & $xyz$ & $\plus$ & $\minus$ & $z$ & $\plus\minus$ & $z\plus$ & $z\minus$ & $z^2$ & $xy$ & $z^2\plus$ & $z^2\minus$ & $z\plus^2$ & $xyz\plus\minus$ & $xy\plus\minus$
\\
$\ZZ_4$ & $1$ & $xyz$ & $\plus$ & $\minus$ & $z$ & $\plus\minus$ & $z\plus$ & $z\minus$ & $z^2$ & $xy$ & $xy\minus$ & $xy\plus$ & $z\plus^2$ & $z^3\plus^2$ & $z^2\plus^2$
\\
$\ZZ_5$ & $1$ & $xy\plus^2$ & $\plus$ & $\minus$ & $\plus^2$ & $\plus\minus$ & $\plus^3$ & $\plus^2\minus$ & $\plus^4$ & $xy$ & $xy\minus$ & $xy\plus$ & $\plus^4$ & $xy\plus^3\minus$ & $xy\plus\minus$
\\
$\ZZ_6$ 		&  $1$& $xy\plus^2$ & $\plus$ & $xy\plus^3$ & $\plus^2$ & $xy\plus^4$ & $\plus^3$ & $xy\plus^5$ & $\plus^4$ & $xy$ & $\plus^5$ & $xy\plus$ & $\plus^4$ & $\plus^8$ & $\plus^6$
\\
$\ZZ_7$ 		&  $1$& $xyz$ & $\plus$ & $xyz\plus$ & $z$ & $xyz^2$ & $z\plus$ & $xyz^2\plus$ & $z^2$ & $xy$ & $z^2\plus$ & $xy\plus$ & $z\plus^2$ & $z^3\plus^2$ & $z^2\plus^2$
\\
$\ZZ_8$		&  $1$& $xyz$ & $\plus$ & $\minus$ & $z$ & $xyz^2$ & $z\plus$ & $z\minus$ & $z^2$ & $xy$ & $z^2\plus$ & $xy\plus$ & $z\plus^2$ & $z^3\plus^2$ & $z^2\plus^2$
\\
$U_9 $		&  $1$& $xyz$ & $\plus$ & $\minus$ & $z$ & $xyz^2$ & $z\plus$ & $z\minus$ & $z^2$ & $xy$ & $z^2\plus$ & $z^2\minus$  & $z\plus^2$ & $z^3\plus^2$ & $z^2\plus^2$
\\
$\ZZ_{10}$ 	&  $1$& $xy\minus^2$ & $\plus$ & $\minus$ & $\minus^2$ & $\plus\minus$ & $\plus\minus^2$ & $\minus^3$ & $\minus^4$ & $xy$ & $xy\minus$ & $xy\plus$ & $\minus^4$ & $xy\plus\minus^3$ & $xy\plus\minus$
\\
$\ZZ_{11}$ 	&  $1$& $xy\minus^2$ & $xy\minus^3$ & $\minus$ & $\minus^2$ & $xy\minus^4$ & $xy\minus^5$ & $\minus^3$ & $\minus^4$ & $xy$ & $xy\minus$ & $\minus^5$ & $\minus^4$ & $\minus^8$ & $\minus^6$
\\
$\ZZ_{12}$ 	& $1$& $xyz$ & $xyz\minus$ & $\minus$ & $z$ & $xyz^2$ & $xyz^2\minus$ & $z\minus$ & $z^2$ & $xy$ & $xy\minus$ & $z^2\minus$ & $z\minus^2$ & $z^3\minus^2$ & $z^2\minus^2$
\\
$\ZZ_{13}$ 	&  $1$& $xyz$ & $\plus$ & $\minus$ & $z$ & $xyz^2$ & $z\plus$ & $z\minus$ & $z^2$ & $xy$ & $xy\minus$ & $z^2\minus$ & $z\minus^2$ & $z^3\minus^2$ & $z^2\minus^2$
\\
$\ZZ_{14}$ 	&  $1$& $xyz$ & $\plus$ & $\minus$ & $z$ & $\plus\minus$ & $z\plus$ & $z\minus$ & $z^2$ & $xy$ & $xy\minus$ & $xy\plus$ & $z\minus^2$ & $xyz\plus\minus$ & $xy\plus\minus$
\\
$\ZZ_{15}$ 	&  $1$& $xyz$ & $\plus$ & $\minus$ & $z$ & $xyz^2$ & $z\plus$ & $z\minus$ & $z^2$ & $xy$ & $z^2\plus$ & $z^2\minus$ & $z\minus^2$ & $z^3\minus^2$ & $z^2\minus^2$
\\
\hline

\end{tabular}
\end{center}
\caption{Basis of the fibres $\ZZ_i$ of each $\RR_\rho$ on every open set $U_i\subset\Hilb{G}{\C^3}$.}
\label{exa:Dn:basisRi}
\end{table}%
}

\begin{prop} Let $G=\Span{\frac{1}{12}(1,7,4),\beta}\subset\SL(3,\C)$, $\pi:Y=\Hilb{G}{\C^3}\to\C^3$ crepant and let $\{\RR_\rho\}_{\rho\in\Irr G}$ be the set of tautological bundles of $Y$. Then, 
\begin{itemize}
\item[(a)] The following relations hold in $\Pic Y$:
\begin{align*}
\RR^-_0 &= \RR^+_4 \otimes \RR^-_8	& \RR^-_6 &= \RR^+_4 \otimes \RR^-_2 & \det\RR_5 = \RR^+_4\otimes\det\RR_9 \\
\RR^+_8 &= \RR^+_4 \otimes \RR^+_4	& \RR^+_6 &= \RR^+_4 \otimes \RR^+_2 &
\end{align*}
\item[(b)] The one-to-one correspondence (\ref{intMcKay}) holds, that is, there exists a bijection between irreducible representations of $G$ and a basis of $H^*(Y,\Z)$.
\end{itemize}
\end{prop}

\begin{proof} $(a)$ The statement comes directly from Table \ref{exa:Dn:basisRi}. It can be computed that the five relations hold in every open set $U_i\subset\Hilb{G}{\C^3}$. 

$(b)$ We know that $\RR_\rho$, for non-trivial $\rho\in\Irr G$, generates $\Pic Y$ but the relations in $(a)$ allows us to reduce the list of generators to 9 by removing from the list $\RR^-_0$, $\RR^+_8$, $\RR^-_6$, $\RR^+_6$ and $\RR_5$. There are 5 relations in $(a)$, precisely the number of compact divisors in $Y$. By \cite{IR} we know that $\rk\Pic Y=|\Irr G\backslash\{\rho_0^+\}|-5= 9$ so the classes of the first Chern class of the 9 remaining tautological bundles form a basis of $H^2(Y,\Z)$. 

The tautological bundle $\RR_0^+\cong\OO_Y$ generates $H^0(Y,\ZZ)$ and, as in the original Reid's recipe, the following virtual bundles
\begin{align*}
\FF^-_0 &:= (\RR^+_4 \otimes \RR^-_8)\ominus(\RR^-_0\otimes\OO_Y)	& \FF^-_6 &:= (\RR^+_4 \otimes \RR^-_2)\ominus(\RR^-_6\otimes\OO_Y)  \\
\FF^+_8 &:= (\RR^+_4 \otimes \RR^+_4)\ominus(\RR^+_8\otimes\OO_Y)	& \FF^+_6 &:= (\RR^+_4 \otimes \RR^+_2)\ominus(\RR^+_6\otimes\OO_Y) \\
\FF_5 &:= (\RR^+_4\otimes\det\RR_9)\ominus(\det\RR_5\otimes\OO_Y)
\end{align*}
give a basis of $H^4(Y,\ZZ)$. For a theoretical proof of the last statement see Craw's argument using the linearization map for $\Pic(Y)$ in the forthcoming \cite{C2}.
\end{proof}

\begin{rem} We conclude the example with some observations that arised from the computations.
\begin{itemize}
\item[(i)] Since the group $\overline{G}\subset\GL(2,\C)$ is a subgroup of $G$, then $\Hilb{\overline{G}}{\C^2}\subset\Hilb{G}{\C^3}$. Indeed, $\Hilb{\overline{G}}{\C^2}$ is recovered as the divisor $z=0$ in the union of open sets $U_1\cup U_5\cup U_6\cup U_{10}\cup U_{11}$. 

\item[(ii)] Every irreducible representation of $G$ appear as either marking a rational curve or marking a divisor $E_i\subset Y$, as it happens in the abelian case (see \cite[4.6]{C}).

\item[(iii)] It can be seeing that the action of $G/A$ on $\Hilb{A}{\C^3}$ also identifies the marking of the abelian subgroup $A$ (see Figure \ref{exa:Ab:marking}). In fact, the relations $\RR^-_6 = \RR^+_4 \otimes \RR^-_2$ and $\RR^+_6 = \RR^+_4 \otimes \RR^+_2$ may be considered as being induced by the relation $\RR_6 = \RR_2\otimes\RR_4$ in $\Pic(\Hilb{A}{\C^3})$, while the rest doesn't. This reflects the fact that in this example (and in general) $\Hilb{G}{\C^3}\ncong\Hilb{G/A}{\Hilb{A}{\C^3}}$, so the marking will also be different.
\end{itemize}
\end{rem}

\section{The trihedral group $G=\Span{\frac{1}{13}(1,3,9),T}\subset\SL(3,\C)$.} \label{exa:T}

A trihedral group $G\subset\SL(3,\C)$ is the semidirect product $G=A\rtimes T$, where $A$ is a finite abelian group and $T$ is the subgroup generated by $T=\left(\begin{smallmatrix}0&1&0\\0&0&1\\1&0&0\end{smallmatrix}\right)$. The case $A=\Z_2\times\Z_2$ is the smallest example of such a group, and the exception fibre $\pi^{-1}(0)\subset\Hilb{G}{\C^3}$ in this case consists of 3 rational curves in a Dynkin $A_3$ configuration, each of them corresponding to the three non-trivial irreducible representations of $G$ (see \cite{NS} for the computations). It is the only trihedral case where this fibre is $1$-dimensional. 

Let us consider from now on the following trihedral group in $\SL(3,\C)$:
\[
G=\Span{\alpha=\frac{1}{13}(1,3,9),T}\subset\SL(3,\C)
\]
The group $G$ is of the form $G=A\rtimes\Z/3\Z$ where $A=\frac{1}{13}(1,3,9)$. It has three 1-dimensional representations $\rho_0$ (trivial), $\rho_0'$ and $\rho_0''$, and four 3-dimensional representations $V_1$, $V_2$, $V_4$ and $V_7$. As in the previous examples, the notation for the representations comes from the action of $T$ on $\Irr A=\{\rho_0,\ldots,\rho_{12}\}$ as follows. The orbits of the action are given by the fixed representation $\rho_0$, which splits into the characters $1$, $\omega$ and $\omega^2$ of the $\Z/3\Z$-action producing $\rho_0$, $\rho_o'$ and $\rho_0''$; and the permutations $(\rho_1,\rho_3,\rho_9)$, $(\rho_2,\rho_6,\rho_5)$, $(\rho_4,\rho_{12},\rho_{10})$ and $(\rho_7,\rho_8,\rho_{11})$ producing $V_1$, $V_2$, $V_4$ and $V_7$ respectively. Let us denote by $\RR_0$, $\RR_0'$, $\RR_0''$, $\RR_1$, $\RR_2$, $\RR_4$ and $\RR_7$ the corresponding tautological bundles.  

The McKay quiver $Q$ with relations is shown in Figure \ref{exa:T13:quiver}, and the maps between the coinvariant algebras $S_\rho$ are:
\[
{\renewcommand{\arraystretch}{1.75}
\begin{array}{c}
a = (x,y,z), b = (x,\omega y, \omega^2z), c = (x,\omega^2y, \omega z), \\
A = \left(\begin{smallmatrix}x\\y\\z\end{smallmatrix}\right), B = \left(\begin{smallmatrix}x\\\omega y\\ \omega^2z\end{smallmatrix}\right), C = \left(\begin{smallmatrix}x\\\omega^2y\\\omega z\end{smallmatrix}\right)\\
d = f = \left(\begin{smallmatrix}0&0&z\\x&0&0\\0&y&0\end{smallmatrix}\right), e = \left(\begin{smallmatrix}y&0&0\\0&z&0\\0&0&x\end{smallmatrix}\right), g = \left(\begin{smallmatrix}x&0&0\\0&y&0\\0&0&z\end{smallmatrix}\right), G = v = \left(\begin{smallmatrix}0&x&0\\0&0&y\\z&0&0\end{smallmatrix}\right), \\
h=\left(\begin{smallmatrix}y&0&0\\0&z&0\\0&0&x\end{smallmatrix}\right), H=u=\left(\begin{smallmatrix}0&0&y\\z&0&0\\0&x&0\end{smallmatrix}\right), i=\left(\begin{smallmatrix}0&z&0\\0&0&x\\y&0&0\end{smallmatrix}\right), j=\left(\begin{smallmatrix}0&0&x\\y&0&0\\0&z&0\end{smallmatrix}\right)
\end{array}
}
\]

\begin{figure}[htbp]
\begin{center}
\begin{tikzpicture}
\node (n1) at (0,0) 
{
\begin{tikzpicture} [scale=0.75]
\node (0) at (0,1.5)  {$0$};
\node (0p) at (1.5,1.5)  {$0'$};
\node (0pp) at (3,1.5)  {$0''$};
\node (1) at (4.5,3)  {$1$};
\node (2) at (7.5,3) {$2$};
\node (4) at (4.5,0)  {$4$};
\node (7) at (7.5,0) {$7$};
\draw [->] (0) to node[gap]  {$\scriptstyle a$} (1);
\draw [->] (0p) to node[gap]  {$\scriptstyle b$} (1);
\draw [->] (0pp) to node[gap]  {$\scriptstyle c$} (1);
\draw[->] (1)+(30:6pt) -- node[above]{$\scriptstyle g$}  ($(2)+(150:6pt)$);
\draw[<-] (1)+(-30:6pt) -- node[below]{$\scriptstyle G$}  ($(2)+(210:6pt)$);
\draw[->] (1)+(-120:6pt) -- node[left]{$\scriptstyle d$}  ($(4)+(120:6pt)$);
\draw[->] (1)+(-60:6pt) -- node[right]{$\scriptstyle e$}  ($(4)+(60:6pt)$);
\draw[->] (4)+(30:6pt) -- node[above]{$\scriptstyle h$}  ($(7)+(150:6pt)$);
\draw[<-] (4)+(-30:6pt) -- node[below]{$\scriptstyle H$}  ($(7)+(210:6pt)$);
\draw [->] (4) to node[gap]  {$\scriptstyle A$} (0);
\draw [->] (4) to node[gap]  {$\scriptstyle B$} (0p);
\draw [->] (4) to node[gap]  {$\scriptstyle C$} (0pp);
\draw [->] (2) to node[right]  {$\scriptstyle f$} (7);
\draw [->] (4) to node[near end, below]  {$\scriptstyle j$} (2);
\draw [->] (7) to node[near end, above]  {$\scriptstyle i$} (1);
\end{tikzpicture}
};
\node (n2) at (6,0.5) {\small$
\begin{array}{ccc}
ad=ae & bd = \omega be & cd = \omega^2ce\\
dA=eA & dB=\omega eB & dC=\omega^2 eC \\
ig=Hj & uG=fi & gu=ej \\
id=vH & jf=hv & gf=dh \\
Ge=fH & Gg=u^2 & Hh=v^2  \\
\end{array}
$
};
\node (n2) at (6,-0.75) {\small$
\begin{array}{c}
Aa+\omega Bb+\omega^2Cc=3hi \\ 
Aa+\omega^2Bb+\omega Cc=3jG \\
\end{array}
$
};
\end{tikzpicture}
\vspace{-0.75cm}
\caption{McKay quiver and relations for $G=\Span{\frac{1}{13}(1,3,9),T}$.}
\label{exa:T13:quiver}
\end{center}
\end{figure}
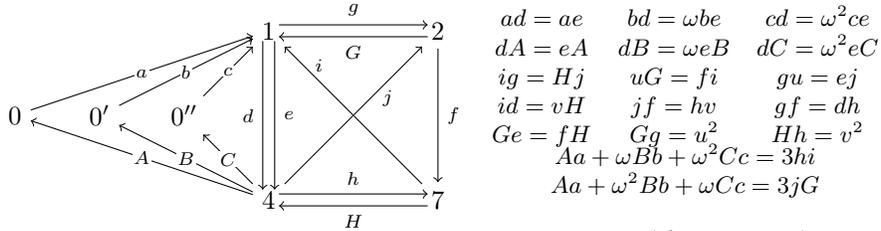

Consider the following polynomials that will be used in what follows:
\[
\begin{array}{ll}
f := x^{13}+y^{13}+z^{13} & g := xy^4+yz^4+x^4z\\
f' := x^{13}+\omega^2y^{13}+\omega z^{13} & g' := xy^4+\omega^2yz^4+\omega x^4z\\
f'' := x^{13}+\omega^2y^{13}+\omega z^{13} & g'' := xy^4+\omega yz^4+\omega^2x^4z
\end{array}
\]
Note that $f,g\in S_{\rho_0}=\C[x,y,z]^G$ are invariant; $f',g'\in S_{\rho_0'}$ and $f'',g''\in S_{\rho_0''}$.

\begin{prop} Let $G=\Span{\frac{1}{13}(1,3,9),T}$, $\pi:Y=\Hilb{G}{\C^3}\to\C^3$ be the crepant resolution and let $\RR_i$ be the tautological bundles of $Y$. 
\begin{itemize}
\item[(a)] There exists an affine open cover of $Y$ consisting of five open sets, i.e.\  $\Hilb{G}{\C^3}\cong\bigcup_{i=1}^5U_i$. The fibre $\pi^{-1}(0)$ consists of two compact divisors $E_1$ and $E_2$ with the configuration shown in Figure \ref{exa:T13:McKay} (left).
\item[(b)] The five open sets in $(a)$ are in one to one correspondence with the five boats of $G$.
\item[(c)] Every nontrivial representation $\rho\in\Irr G$ marks either a curve or a compact divisor in $Y$ (see Figure \ref{exa:T13:McKay} right).
\item[(d)] The following two relations hold in $\Pic Y$:
		\[
		\RR_{0}'\otimes\RR_{0}''= \det\RR_4  \hspace{1cm} \det\RR_1=\det\RR_2
		\]
\item[(e)] The one-to-one correspondence (\ref{intMcKay}) holds, that is, there exists a bijection between irreducible representations of $G$ and a basis of $H^*(Y,\Z)$.
\end{itemize}
\end{prop}

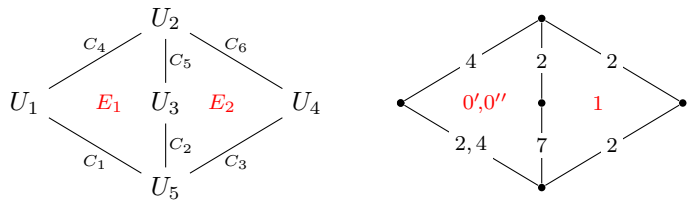
\begin{figure}[htbp]
\begin{center}

\begin{tikzpicture}

\node at (0,0) {
\begin{tikzpicture}[scale=0.75]
\node (u1) at (0,1.5) {$U_1$};
\node (u2) at (2.5,3) {$U_2$};
\node (u3) at (2.5,1.5) {$U_3$};
\node (u4) at (5,1.5) {$U_4$};
\node (u5) at (2.5,0) {$U_5$};
\draw[-] (u1) to node[above] {\tiny$C_4$} (u2);
\draw[-] (u2) to node[right,inner sep=1pt] {\tiny$C_5$} (u3);
\draw[-] (u3) to node[right,inner sep=1pt] {\tiny$C_2$} (u5);
\draw[-] (u1) to node[below] {\tiny$C_1$} (u5);
\draw[-] (u2) to node[above] {\tiny$C_6$} (u4);
\draw[-] (u5) to node[below] {\tiny$C_3$} (u4);
\node at (1.5,1.5) [red] {\footnotesize $E_1$};
\node at (3.5,1.5) [red] {\footnotesize $E_2$};
\end{tikzpicture}
};

\node at (5,0) {
\begin{tikzpicture}[scale=0.75]
\node (u1) [circle,fill=black,inner sep=1pt] at (0,1.5) {};
\node (u1) [circle,fill=black,inner sep=1pt] at (0,1.5) {};
\node (u2) [circle,fill=black,inner sep=1pt] at (2.5,3) {};
\node (u3) [circle,fill=black,inner sep=1pt] at (2.5,1.5) {};
\node (u4) [circle,fill=black,inner sep=1pt] at (5,1.5) {};
\node (u5) [circle,fill=black,inner sep=1pt] at (2.5,0) {};
\draw[-] (u1) to node[gap,inner sep=2pt] {\footnotesize$4$} (u2);
\draw[-] (u2) to node[gap,inner sep=2pt] {\footnotesize$2$} (u3);
\draw[-] (u3) to node[gap,inner sep=2pt] {\footnotesize$7$} (u5);
\draw[-] (u1) to node[gap,inner sep=2pt] {\footnotesize$2,4$} (u5);
\draw[-] (u2) to node[gap,inner sep=2pt] {\footnotesize$2$} (u4);
\draw[-] (u5) to node[gap,inner sep=2pt] {\footnotesize$2$} (u4);
\node at (1.5,1.5) [red] {\footnotesize $0'\!,\!0''$};
\node at (3.5,1.5) [red] {\footnotesize $1$};
\end{tikzpicture}
};

\end{tikzpicture}
\vspace{-0.5cm}
\caption{Description of the exceptional components of $\pi^{-1}(0)\subset\Hilb{G}{\C^3}$ (left) and its marking (right).}
\label{exa:T13:McKay}
\end{center}
\end{figure}

\begin{proof} (a) Let us denote by $d_i$ for $i=1,2,3$ the $i$-th row of the matrix $d$, and use the analogous notation for the rest of the matrices in the representation space of the McKay quiver. Then, the linearly independent paths defining the five open sets $U_i\subset\Hilb{G}{\C^3}$ are shown in Table \ref{exa:T:OpenConds}.

{\small
\begin{table}[h]
\begin{center}
\begin{tabular}{cl}
Open & Linearly independent paths in $Q$ (open conditions) \\
\hline
$U_1$ 	&  $a=(1,0,0)$, $B_2=C_2=1$, $e_3=g_1=h_1=(1,0,0)$, \\
		& $f_2=G_1=H_1=u_3=(0,1,0)$, $d_2=i_3=j_3=v_2=(0,0,1)$ \\
\hline
$U_2$ 	&   $a=(1,0,0)$, $B_3=C_3=1$, $d_1=g_1=h_1=(1,0,0)$, \\
		& $e_2=f_2=G_1=u_3=(0,1,0)$, $d_2=i_3=j_3=v_2=(0,0,1)$ \\
\hline
$U_3$ 	& $a=(1,0,0)$, $B_3=C_3=1$, $d_1=g_1=h_1=(1,0,0)$, \\
		& $e_2=f_2=G_1=u_1=(0,1,0)$, $d_2=g_2=G_3=H_3=(0,0,1)$ \\
\hline
$U_4$ 	& $a=(1,0,0)$, $B_3=C_3=1$, $d_1=g_1=h_1=(1,0,0)$, \\
		& $e_2=f_2=G_1=u_2=(0,1,0)$, $b=d_2=j_3=v_2=(0,0,1)$ \\
\hline
$U_5$ 	& $a=(1,0,0)$, $B_3=C_3=1$, $d_1=g_1=h_1=(1,0,0)$, \\
		& $e_2=f_2=G_1=u_2=(0,1,0)$, $d_2=g_2=G_3=v_2=(0,0,1)$ \\
\hline
\end{tabular}
\end{center}
\caption{Open conditions for each open set $U_i\subset\Hilb{G}{\C^3}$, for $i=1,\ldots,5$.}
\label{exa:T:OpenConds}
\end{table}%
}

As a result of the choices made in Table \ref{exa:T:OpenConds} we obtain $U_5\cong\C^3$; $U_4$ and $U_3$ defined as hypersurfaces in $\C^4$; and the open sets $U_1$ and $U_2$ defined by a large number of equations and variables. Even though the explicit description of the last two open sets is not ideal, it is possible to restrict the calculations to the exceptional divisor obtaining the two compact divisors $E_1$ and $E_2$, and five rational curves $C_1,\ldots,C_6$ as shown in Figure \ref{exa:T13:McKay} (left).

(b) In order to construct explicitly $\Hilb{G}{\C^3}$ for trihedral groups M. Reid has introduced the notion of {\em trihedral boats} in order to describe a monomial basis for every $G$-cluster, following Nakamura's spirit of $A$-graphs in the abelian case \footnote{See \url{https://homepages.warwick.ac.uk/~masda/McKay/tri/} for information, references and scripts about boats}. A trihedral boat is a subset of $|A|$ monomials, one in each irreducible representation of $A$, which is a fundamental domain for the action of $T$ on a $G$-cluster. In our case, the boats $\mathcal{B}_i$ corresponding to each open set $U_i$, $i=1,\ldots,5$, are shown in Figure \ref{exa:T13:boats}.

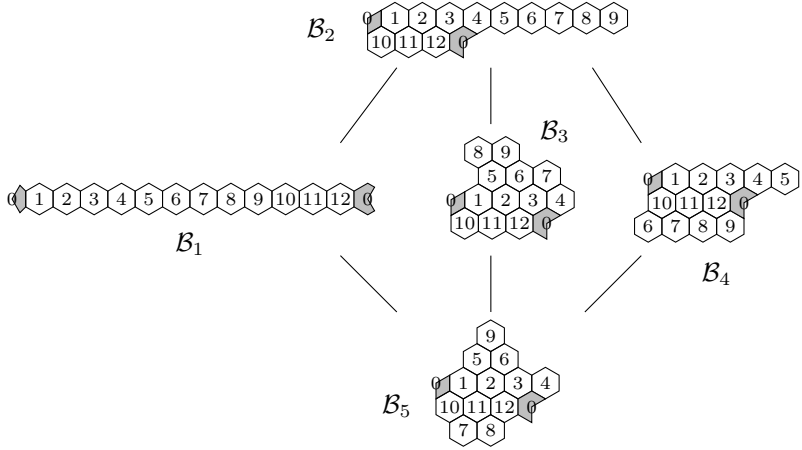
\begin{figure}[htbp]
\begin{center}
\begin{tikzpicture}

\node at (0,1.65) {$\mathcal{B}_1$};
\node at (2.75,-0.5) {$\mathcal{B}_5$};
\node at (7,1.25) {$\mathcal{B}_4$};
\node at (4.85,3.15) {$\mathcal{B}_3$};
\node at (1.75,4.5) {$\mathcal{B}_2$};

\draw (4,0.75) -- (4,1.5);
\draw (2.75,0.75) -- (2,1.5);
\draw (5.25,0.75) -- (6,1.5);
\draw (4,3.25) -- (4,4);
\draw (2,3) -- (2.75,4);
\draw (6,3) -- (5.35,4);

\node at (0,2.25) {
\begin{tikzpicture}[scale=0.3,bend angle=45, looseness=1]
	\draw [fill=lightgray] (0,0) -- (60:0.7) -- (30:0.7) -- (-30:0.7) -- (-60:0.7) -- cycle; \node at (0,0) {\scriptsize$0$};
	\hex{((0,0)}{(\r,0)} \node at (\r,0) {\scriptsize$1$};
	\hex{((0,0)}{(2*\r,0)} \node at (2*\r,0) {\scriptsize$2$};
	\hex{((0,0)}{(3*\r,0)} \node at (3*\r,0) {\scriptsize$3$};
	\hex{((0,0)}{(4*\r,0)} \node at (4*\r,0) {\scriptsize$4$};
	\hex{((0,0)}{(5*\r,0)} \node at (5*\r,0) {\scriptsize$5$};
	\hex{((0,0)}{(6*\r,0)} \node at (6*\r,0) {\scriptsize$6$};
	\hex{((0,0)}{(7*\r,0)} \node at (7*\r,0) {\scriptsize$7$};
	\hex{((0,0)}{(8*\r,0)} \node at (8*\r,0) {\scriptsize$8$};
	\hex{((0,0)}{(9*\r,0)} \node at (9*\r,0) {\scriptsize$9$};
	\hex{((0,0)}{(10*\r,0)} \node at (10*\r,0) {\scriptsize$10$};
	\hex{((0,0)}{(11*\r,0)} \node at (11*\r,0) {\scriptsize$11$};
	\hex{((0,0)}{(12*\r,0)} \node at (12*\r,0) {\scriptsize$12$};
	\draw [fill=lightgray] (12.5*\r,0.7/2) -- (13*\r,0.7) -- (13.25*\r,3/4*0.7) -- (13*\r,0) -- (13.25*\r,-3/4*0.7) -- (13*\r,-0.7) -- (12.5*\r,-0.7/2) -- cycle;
	\node at (13*\r,0) {\scriptsize$0$};
\end{tikzpicture}
};

\node at (4,4.5) {
\begin{tikzpicture}[scale=0.3,bend angle=45, looseness=1]
	\draw [fill=lightgray] (0,0) -- (30:0.7) -- (-30:0.7) -- (-90:0.7) -- cycle; \node at (0,0) {\scriptsize$0$};
	\hex{((0,0)}{(\r,0)} \node at (\r,0) {\scriptsize$1$};
	\hex{((0,0)}{(2*\r,0)} \node at (2*\r,0) {\scriptsize$2$};
	\hex{((0,0)}{(3*\r,0)} \node at (3*\r,0) {\scriptsize$3$};
	\hex{((0,0)}{(4*\r,0)} \node at (4*\r,0) {\scriptsize$4$};
	\hex{((0,0)}{(5*\r,0)} \node at (5*\r,0) {\scriptsize$5$};
	\hex{((0,0)}{(6*\r,0)} \node at (6*\r,0) {\scriptsize$6$};
	\hex{((0,0)}{(7*\r,0)} \node at (7*\r,0) {\scriptsize$7$};
	\hex{((0,0)}{(8*\r,0)} \node at (8*\r,0) {\scriptsize$8$};
	\hex{((0,0)}{(9*\r,0)} \node at (9*\r,0) {\scriptsize$9$};
	\hex{((0,0)}{(0.5*\r,-1.5*0.7)} \node at (0.5*\r,-1.5*0.7) {\scriptsize$10$};
	\hex{((0,0)}{(1.5*\r,-1.5*0.7)} \node at (1.5*\r,-1.5*0.7) {\scriptsize$11$};
	\hex{((0,0)}{(2.5*\r,-1.5*0.7)} \node at (2.5*\r,-1.5*0.7) {\scriptsize$12$};
	\draw [fill=lightgray] (4*\r,-0.7) -- (3.5*\r,-1.5*0.7) -- (3.5*\r,-2.5*0.7) -- (3*\r,-2*0.7) -- (3*\r,-0.7) -- (3.5*\r,-0.5*0.7) --cycle;
	\node at (3.5*\r,-1.5*0.7) {\scriptsize$0$};
\end{tikzpicture}
};

\node at (4.2,2.4) {
\begin{tikzpicture}[scale=0.3,bend angle=45, looseness=1]
	\draw [fill=lightgray] (0,0) -- (30:0.7) -- (-30:0.7) -- (-90:0.7) -- cycle; \node at (0,0) {\scriptsize$0$};
	\hex{((0,0)}{(\r,0)} \node at (\r,0) {\scriptsize$1$};
	\hex{((0,0)}{(2*\r,0)} \node at (2*\r,0) {\scriptsize$2$};
	\hex{((0,0)}{(3*\r,0)} \node at (3*\r,0) {\scriptsize$3$};
	\hex{((0,0)}{(4*\r,0)} \node at (4*\r,0) {\scriptsize$4$};
	
	\hex{((0,0)}{(1.5*\r,1.5*0.7)} \node at (1.5*\r,1.5*0.7) {\scriptsize$5$};
	\hex{((0,0)}{(2.5*\r,1.5*0.7)} \node at (2.5*\r,1.5*0.7) {\scriptsize$6$};
	\hex{((0,0)}{(3.5*\r,1.5*0.7)} \node at (3.5*\r,1.5*0.7) {\scriptsize$7$};
	
	\hex{((0,0)}{(\r,3*0.7)} \node at (\r,3*0.7) {\scriptsize$8$};
	\hex{((0,0)}{(2*\r,3*0.7)} \node at (2*\r,3*0.7) {\scriptsize$9$};
	
	\hex{((0,0)}{(0.5*\r,-1.5*0.7)} \node at (0.5*\r,-1.5*0.7) {\scriptsize$10$};
	\hex{((0,0)}{(1.5*\r,-1.5*0.7)} \node at (1.5*\r,-1.5*0.7) {\scriptsize$11$};
	\hex{((0,0)}{(2.5*\r,-1.5*0.7)} \node at (2.5*\r,-1.5*0.7) {\scriptsize$12$};
	\draw [fill=lightgray] (4*\r,-0.7) -- (3.5*\r,-1.5*0.7) -- (3.5*\r,-2.5*0.7) -- (3*\r,-2*0.7) -- (3*\r,-0.7) -- (3.5*\r,-0.5*0.7) --cycle;
	\node at (3.5*\r,-1.5*0.7) {\scriptsize$0$};
\end{tikzpicture}
};

\node at (7,2.2) {
\begin{tikzpicture}[scale=0.3,bend angle=45, looseness=1]
	\draw [fill=lightgray] (0,0) -- (30:0.7) -- (-30:0.7) -- (-90:0.7) -- cycle; \node at (0,0) {\scriptsize$0$};
	\hex{((0,0)}{(\r,0)} \node at (\r,0) {\scriptsize$1$};
	\hex{((0,0)}{(2*\r,0)} \node at (2*\r,0) {\scriptsize$2$};
	\hex{((0,0)}{(3*\r,0)} \node at (3*\r,0) {\scriptsize$3$};
	\hex{((0,0)}{(4*\r,0)} \node at (4*\r,0) {\scriptsize$4$};
	\hex{((0,0)}{(5*\r,0)} \node at (5*\r,0) {\scriptsize$5$};
		
	\hex{((0,0)}{(0*\r,-3*0.7)} \node at (0*\r,-3*0.7) {\scriptsize$6$};
	\hex{((0,0)}{(1*\r,-3*0.7)} \node at (1*\r,-3*0.7) {\scriptsize$7$};
	\hex{((0,0)}{(2*\r,-3*0.7)} \node at (2*\r,-3*0.7) {\scriptsize$8$};
	\hex{((0,0)}{(3*\r,-3*0.7)} \node at (3*\r,-3*0.7) {\scriptsize$9$};
	
	\hex{((0,0)}{(0.5*\r,-1.5*0.7)} \node at (0.5*\r,-1.5*0.7) {\scriptsize$10$};
	\hex{((0,0)}{(1.5*\r,-1.5*0.7)} \node at (1.5*\r,-1.5*0.7) {\scriptsize$11$};
	\hex{((0,0)}{(2.5*\r,-1.5*0.7)} \node at (2.5*\r,-1.5*0.7) {\scriptsize$12$};
	\draw [fill=lightgray] (4*\r,-0.7) -- (3.5*\r,-1.5*0.7) -- (3.5*\r,-2.5*0.7) -- (3*\r,-2*0.7) -- (3*\r,-0.7) -- (3.5*\r,-0.5*0.7) --cycle;
	\node at (3.5*\r,-1.5*0.7) {\scriptsize$0$};
\end{tikzpicture}
};

\node at (4,-0.2) {
\begin{tikzpicture}[scale=0.3,bend angle=45, looseness=1]
	\draw [fill=lightgray] (0,0) -- (30:0.7) -- (-30:0.7) -- (-90:0.7) -- cycle; \node at (0,0) {\scriptsize$0$};
	\hex{((0,0)}{(\r,0)} \node at (\r,0) {\scriptsize$1$};
	\hex{((0,0)}{(2*\r,0)} \node at (2*\r,0) {\scriptsize$2$};
	\hex{((0,0)}{(3*\r,0)} \node at (3*\r,0) {\scriptsize$3$};
	\hex{((0,0)}{(4*\r,0)} \node at (4*\r,0) {\scriptsize$4$};

	\hex{((0,0)}{(1.5*\r,1.5*0.7)} \node at (1.5*\r,1.5*0.7) {\scriptsize$5$};	
	\hex{((0,0)}{(2.5*\r,1.5*0.7)} \node at (2.5*\r,1.5*0.7) {\scriptsize$6$};
	
	\hex{((0,0)}{(\r,-3*0.7)} \node at (\r,-3*0.7) {\scriptsize$7$};
	\hex{((0,0)}{(2*\r,-3*0.7)} \node at (2*\r,-3*0.7) {\scriptsize$8$};

	\hex{((0,0)}{(2*\r,3*0.7)} \node at (2*\r,3*0.7) {\scriptsize$9$};
	
	\hex{((0,0)}{(0.5*\r,-1.5*0.7)} \node at (0.5*\r,-1.5*0.7) {\scriptsize$10$};
	\hex{((0,0)}{(1.5*\r,-1.5*0.7)} \node at (1.5*\r,-1.5*0.7) {\scriptsize$11$};
	\hex{((0,0)}{(2.5*\r,-1.5*0.7)} \node at (2.5*\r,-1.5*0.7) {\scriptsize$12$};
	\draw [fill=lightgray] (4*\r,-0.7) -- (3.5*\r,-1.5*0.7) -- (3.5*\r,-2.5*0.7) -- (3*\r,-2*0.7) -- (3*\r,-0.7) -- (3.5*\r,-0.5*0.7) --cycle;
	\node at (3.5*\r,-1.5*0.7) {\scriptsize$0$};
\end{tikzpicture}
};

\end{tikzpicture}
\vspace{-0.5cm}
\caption{The five boats for the trihedral group $\Span{\frac{1}{13}(1,3,9),T}$.}
\label{exa:T13:boats}
\end{center}
\end{figure}

The plane of hexagonal cells represent every monomial in $\C[x,y,z]$ modulo $xyz$, and the numbers indicate the irreducible representation $\rho_i\in\Irr A$ in which they belong. For example, the numbers 0, 1, 2 , 3 and 4 contained in the consecutive hexagons that can be found in every boat, represent the monomials $1$, $x$, $x^2$, $x^3$ and $x^4$ which belong to $\rho_0$, $\rho_1$, $\rho_2$, $\rho_3$ and $\rho_4$ respectively. The point is that if we rotate a boat $\mathcal{B}$ twice an angle of $2\pi/3$ around the hexagon 0 with the monomial 1 (which is equivalent to the permutation defined by $T:x\mapsto y\mapsto z\mapsto x$), the monomials of the shape formed by the 3 pieces together $\mathcal{B}\cup T\mathcal{B} \cup T^2\mathcal{B}$ form a $G$-graph, thus a monomial basis of every $G$-cluster in an affine open subset of $\Hilb{G}{\C^3}$. Indeed, every boat contains one element in each representation $\rho_i\in\Irr A$, with $i\neq0$, and the orbits $(\rho_1,\rho_3,\rho_9)$, $(\rho_2,\rho_6,\rho_5)$, $(\rho_4,\rho_{12},\rho_{10})$ and $(\rho_7,\rho_8,\rho_{11})$ for the action of $T$ form 3 elements in each irreducible representation $V_1$, $V_2$, $V_4$, $V_7\in\Irr G$ (see the columns on the matrices in Table \ref{exa:T:basisRi}). For $\rho_0$, every boat has one hexagon although it is divided in two parts (highlighted in gray). Once we join the three pieces acting by $T$, we obtain 3 full hexagons corresponding to three polynomials $p=1$, $p'$ and $p''$ belonging to the 1-dimensional representations $\rho_0$, $\rho_0'$, $\rho_0''\in\Irr G$. For example, in $\mathcal{B}_5$ we have monomials $1$ and $x^4z$ in the boat. After joining the 3 boats we obtain the monomials
\[
1, x^4z, xy^4, yz^4 \in S_{\rho_0\in\Irr A}
\]
which give the polynomials
\begin{align*}
1 &\in S_{\rho_0}\\
x^4z +\omega xy^4+\omega^2 yz^4 &\in S_{\rho_0'}\\
x^4z +\omega^2 xy^4+\omega yz^4 &\in S_{\rho_0''}
\end{align*}
This phenomenon is called {\em twinning}.

The boats $\mathcal{B}_i$ are obtained from the elements of each $\RR_i$ listed in Table \ref{exa:T:basisRi}, which is the result of the open conditions in Table \ref{exa:T:OpenConds} for each $U_1$, $i=1,\ldots,5$. The columns of the matrices in Table \ref{exa:T:basisRi} for the bundles $\RR_1$ for $i=1,2,4,7$ are the basis elements for each of them (recall that their rank is three).

{\scriptsize
\begin{table}[htp]
\begin{center}
\setlength{\tabcolsep}{1pt}
\renewcommand{\arraystretch}{3}
\begin{tabular}{ccccccccc}
 & $\RR_0$ & $\RR_{0'}$ & $\RR_{0''}$ & $\RR_1$ & $\RR_2$ & $\RR_4$ & $\RR_7$ & Socle at $(0,0,0)$\\
\hline
$\ZZ_1$ &  $1$ & $f'$ & $f''$  
	& $\left[\begin{smallmatrix}x&z^3&y^9\\y&x^3&z^9\\x&y^3&x^9\end{smallmatrix}\right]$ 
	& $\left[\begin{smallmatrix}x^2&z^6&y^5\\y^2&x^6&z^5\\z^2&y^6&x^5\end{smallmatrix}\right]$
	& $\left[\begin{smallmatrix}y^{10}&z^{12}&x^4\\z^{10}&x^{12}&y^4\\x^{10}&y^{12}&z^4\end{smallmatrix}\right]$
	& $\left[\begin{smallmatrix}y^{11}&x^7&z^8\\z^{11}&y^7&x^8\\x^{11}&z^7&y^8\end{smallmatrix}\right]$ & $\rho_0', \rho_0''$ \\ 
$\ZZ_2$ &  $1$ & $g'$ & $g''$  	
	& $\left[\begin{smallmatrix}x&z^3&y^9\\y&x^3&z^9\\z&y^3&x^9\end{smallmatrix}\right]$ 
	& $\left[\begin{smallmatrix}x^2&z^6&y^5\\y^2&x^6&z^5\\z^2&y^6&x^5\end{smallmatrix}\right]$
	& $\left[\begin{smallmatrix}xy&yz^3&x^4\\yz&x^3z&y^4\\xz&xy^3&z^4\end{smallmatrix}\right]$
	& $\left[\begin{smallmatrix}xy^2&x^7&z^8\\yz^2&y^7&x^8\\x^2z&z^7&y^8\end{smallmatrix}\right]$ & $\rho_0', \rho_0'', V_1$ \\ 
$\ZZ_3$ &  $1$ & $g'$ & $g''$   	
	& $\left[\begin{smallmatrix}x&z^3&y^3z^2\\y&x^3&x^2z^3\\z&y^3&x^3y^2\end{smallmatrix}\right]$ 
	& $\left[\begin{smallmatrix}x^2&y^2z&xz^3\\y^2&xz^2&x^3y\\z^2&x^2y&y^3z\end{smallmatrix}\right]$
	& $\left[\begin{smallmatrix}xy&yz^3&x^4\\yz&x^3z&y^4\\xz&xy^3&z^4\end{smallmatrix}\right]$
	& $\left[\begin{smallmatrix}xy^2&x^2z^2&x^4y\\yz^2&x^2y^2&y^4z\\x^2z&y^2z^2&xz^4\end{smallmatrix}\right]$ & $\rho_0', \rho_0'', V_1, V_7$ \\ 
$\ZZ_4$ &  $1$ & $g'$ & $g''$   
	& $\left[\begin{smallmatrix}x&z^3&\\y&x^3&r\\z&y^3&\end{smallmatrix}\right]$ 
	& $\left[\begin{smallmatrix}x^2&y^2z&y^5\\y^2&xz^2&z^5\\z^2&x^2y&x^5\end{smallmatrix}\right]$
	& $\left[\begin{smallmatrix}xy&yz^3&x^4\\yz&x^3z&y^4\\xz&xy^3&z^4\end{smallmatrix}\right]$
	& $\left[\begin{smallmatrix}xy^2&x^2z^2&y^2z^3\\yz^2&x^2y^2&x^3z^2\\x^2z&y^2z^2&x^2y^3\end{smallmatrix}\right]$ & $V_1$  \\ 
$\ZZ_5$ &  $1$ & $g'$ & $g''$   	
	& $\left[\begin{smallmatrix}x&z^3&y^3z^2\\y&x^3&x^2z^3\\z&y^3&x^3y^2\end{smallmatrix}\right]$ 
	& $\left[\begin{smallmatrix}x^2&y^2z&xz^3\\y^2&xz^2&x^3y\\z^2&x^2y&y^3z\end{smallmatrix}\right]$
	& $\left[\begin{smallmatrix}xy&yz^3&x^4\\yz&x^3z&y^4\\xz&xy^3&z^4\end{smallmatrix}\right]$
	& $\left[\begin{smallmatrix}xy^2&x^2z^2&y^2z^3\\yz^2&x^2y^2&x^3z^2\\x^2z&y^2z^2&x^2y^3\end{smallmatrix}\right]$ & $\rho_0', \rho_0'', V_1, V_7$ \\ 
\hline
\end{tabular}
\end{center}
\caption{Basis of the fibres $\ZZ_i$ of $\RR$ on every open set in $U_i\subset\Hilb{G}{\C^3}$. The basis for $\RR_1$ in $U_4$ is $r:=(x^2y^4,y^2+z^4,x^4z^2)+\omega(x^5z,xy^5,yz^5)+\omega^2(xyz^4,x^4yz,xy^4z)$, corresponding to the path $agGdBb$ in $Q$.}
\label{exa:T:basisRi}
\end{table}%
}

(c) The computation of the socles at the origin of every open set $U_i\subset\Hilb{G}{\C^3}$ is shown in Table \ref{exa:T:basisRi}. From it we mark $E_1$ with representations $\rho_0',\rho_0''$, and we mark $E_2$ with the representation $V_1$.

The marking of the curves $C_i\subset\Hilb{G}{\C^3}$, $i=1,\ldots,6$, uses the $G$-igsaw transformation modules computed in Table \ref{exa:T:GigMod}. Notice that $C_1$ is marked by two different representations. 

{\small
\begin{table}[htp]
\begin{center}
\setlength{\tabcolsep}{1pt}
\renewcommand{\arraystretch}{1.5}
\begin{tabular}{clc}
  & $G$-igsaw transformation modules $G$-igMod$_{C_i}$. & Marking \\
\hline
$C_1~~~~~$ & $G$-igMod$_{C_1}(U_1)=\Span{y^5,z^5,x^5}\in V_2$, $G$-igMod$_{C_1}(U_5)=\Span{xy,yz,xz}\in V_4$  & $V_2, V_4$\\ 
$C_2~~~~~$ & $G$-igMod$_{C_2}(U_3)=\Span{y^2z^3,x^3z^2,x^2y^3}\in V_7$, $G$-igMod$_{C_2}(U_5)=\Span{x^4y,y^4z,xz^4}\in V_7$ & $V_7$ \\
$C_3~~~~~$ & $G$-igMod$_{C_3}(U_4)=\Span{xz^3,x^3y,y^3z}\in V_2$, $G$-igMod$_{C_3}(U_5)=\Span{y^5,z^5,x^5}\in V_2$ & $V_2$ \\ 
$C_4~~~~~$ & $G$-igMod$_{C_4}(U_1)=\Span{xy,yz,xz}\in V_4$, $G$-igMod$_{C_4}(U_2)=\Span{y^{10},z^{10},x^{10}}\in V_4$ & $V_4$ \\ 
$C_5~~~~~$ & $G$-igMod$_{C_5}(U_2)=\Span{y^5,z^5,x^5}\in V_2$, $G$-igMod$_{C_5}(U_3)=\Span{y^2z,xz^2,x^2y}\in V_2$ & $V_2$ \\ 
$C_6~~~~~$ & $G$-igMod$_{C_6}(U_2)=\Span{y^2z,xz^2,x^2y}\in V_2$, $G$-igMod$_{C_6}(U_4)=\Span{z^6,x^6,y^6}\in V_2$ & $V_2$ \\ 
\hline
\end{tabular}
\end{center}
\caption{Marking of the curves $C_i\subset\Hilb{G}{\C^3}$ for $i=1,\ldots,6$.}
\label{exa:T:GigMod}
\end{table}%
}

(d) Figure \ref{exa:T13:degrees} shows the degrees of the tautological bundles $\RR_0'$, $\RR_0''$, $\det(\RR_1)$, $\det(\RR_2)$, $\det(\RR_4)$ and $\det(\RR_7)$ at every rational curve of $\Hilb{G}{\C^3}$. From them we can say that since $\det(\RR_1)$ and $\det(\RR_2)$ have the same degree on every compact curve on $Y$ then $\det(\RR_1)=\det(\RR_2)$ in $\Pic Y$. Similarly, we can conclude that $\det(\RR_4)=\RR_0'\otimes\RR_0''$.

\begin{figure}[htbp]
\begin{center}

\begin{tikzpicture}

\node at (0,1.75) {$\RR_0'$};
\node at (0,3) {
\begin{tikzpicture}[scale=0.5]
\node (u1) [circle,fill=black,inner sep=1pt] at (0,1.5) {};
\node (u1) [circle,fill=black,inner sep=1pt] at (0,1.5) {};
\node (u2) [circle,fill=black,inner sep=1pt] at (2.5,3) {};
\node (u3) [circle,fill=black,inner sep=1pt] at (2.5,1.5) {};
\node (u4) [circle,fill=black,inner sep=1pt] at (5,1.5) {};
\node (u5) [circle,fill=black,inner sep=1pt] at (2.5,0) {};
\draw[-] (u1) to node[gap,inner sep=2pt] {\footnotesize$1$} (u2);
\draw[-] (u2) to node[gap,inner sep=2pt] {\footnotesize$0$} (u3);
\draw[-] (u3) to node[gap,inner sep=2pt] {\footnotesize$0$} (u5);
\draw[-] (u1) to node[gap,inner sep=2pt] {\footnotesize$4$} (u5);
\draw[-] (u2) to node[gap,inner sep=2pt] {\footnotesize$0$} (u4);
\draw[-] (u5) to node[gap,inner sep=2pt] {\footnotesize$0$} (u4);
\end{tikzpicture}
};

\node at (4,1.75) {$\RR_0''$};
\node at (4,3) {
\begin{tikzpicture}[scale=0.5]
\node (u1) [circle,fill=black,inner sep=1pt] at (0,1.5) {};
\node (u1) [circle,fill=black,inner sep=1pt] at (0,1.5) {};
\node (u2) [circle,fill=black,inner sep=1pt] at (2.5,3) {};
\node (u3) [circle,fill=black,inner sep=1pt] at (2.5,1.5) {};
\node (u4) [circle,fill=black,inner sep=1pt] at (5,1.5) {};
\node (u5) [circle,fill=black,inner sep=1pt] at (2.5,0) {};
\draw[-] (u1) to node[gap,inner sep=2pt] {\footnotesize$1$} (u2);
\draw[-] (u2) to node[gap,inner sep=2pt] {\footnotesize$0$} (u3);
\draw[-] (u3) to node[gap,inner sep=2pt] {\footnotesize$0$} (u5);
\draw[-] (u1) to node[gap,inner sep=2pt] {\footnotesize$4$} (u5);
\draw[-] (u2) to node[gap,inner sep=2pt] {\footnotesize$0$} (u4);
\draw[-] (u5) to node[gap,inner sep=2pt] {\footnotesize$0$} (u4);
\end{tikzpicture}
};

\node at (8,1.75) {$\det(\RR_1)$};
\node at (8,3) {
\begin{tikzpicture}[scale=0.5]
\node (u1) [circle,fill=black,inner sep=1pt] at (0,1.5) {};
\node (u1) [circle,fill=black,inner sep=1pt] at (0,1.5) {};
\node (u2) [circle,fill=black,inner sep=1pt] at (2.5,3) {};
\node (u3) [circle,fill=black,inner sep=1pt] at (2.5,1.5) {};
\node (u4) [circle,fill=black,inner sep=1pt] at (5,1.5) {};
\node (u5) [circle,fill=black,inner sep=1pt] at (2.5,0) {};
\draw[-] (u1) to node[gap,inner sep=2pt] {\footnotesize$0$} (u2);
\draw[-] (u2) to node[gap,inner sep=2pt] {\footnotesize$1$} (u3);
\draw[-] (u3) to node[gap,inner sep=2pt] {\footnotesize$0$} (u5);
\draw[-] (u1) to node[gap,inner sep=2pt] {\footnotesize$2$} (u5);
\draw[-] (u2) to node[gap,inner sep=2pt] {\footnotesize$1$} (u4);
\draw[-] (u5) to node[gap,inner sep=2pt] {\footnotesize$1$} (u4);
\end{tikzpicture}
};

\node at (0,-1.25) {$\det(\RR_2)$};
\node at (0,0) {
\begin{tikzpicture}[scale=0.5]
\node (u1) [circle,fill=black,inner sep=1pt] at (0,1.5) {};
\node (u1) [circle,fill=black,inner sep=1pt] at (0,1.5) {};
\node (u2) [circle,fill=black,inner sep=1pt] at (2.5,3) {};
\node (u3) [circle,fill=black,inner sep=1pt] at (2.5,1.5) {};
\node (u4) [circle,fill=black,inner sep=1pt] at (5,1.5) {};
\node (u5) [circle,fill=black,inner sep=1pt] at (2.5,0) {};
\draw[-] (u1) to node[gap,inner sep=2pt] {\footnotesize$0$} (u2);
\draw[-] (u2) to node[gap,inner sep=2pt] {\footnotesize$1$} (u3);
\draw[-] (u3) to node[gap,inner sep=2pt] {\footnotesize$0$} (u5);
\draw[-] (u1) to node[gap,inner sep=2pt] {\footnotesize$2$} (u5);
\draw[-] (u2) to node[gap,inner sep=2pt] {\footnotesize$1$} (u4);
\draw[-] (u5) to node[gap,inner sep=2pt] {\footnotesize$1$} (u4);
\end{tikzpicture}
};

\node at (4,-1.25) {$\det(\RR_4)$};
\node at (4,0) {
\begin{tikzpicture}[scale=0.5]
\node (u1) [circle,fill=black,inner sep=1pt] at (0,1.5) {};
\node (u1) [circle,fill=black,inner sep=1pt] at (0,1.5) {};
\node (u2) [circle,fill=black,inner sep=1pt] at (2.5,3) {};
\node (u3) [circle,fill=black,inner sep=1pt] at (2.5,1.5) {};
\node (u4) [circle,fill=black,inner sep=1pt] at (5,1.5) {};
\node (u5) [circle,fill=black,inner sep=1pt] at (2.5,0) {};
\draw[-] (u1) to node[gap,inner sep=2pt] {\footnotesize$2$} (u2);
\draw[-] (u2) to node[gap,inner sep=2pt] {\footnotesize$0$} (u3);
\draw[-] (u3) to node[gap,inner sep=2pt] {\footnotesize$0$} (u5);
\draw[-] (u1) to node[gap,inner sep=2pt] {\footnotesize$8$} (u5);
\draw[-] (u2) to node[gap,inner sep=2pt] {\footnotesize$0$} (u4);
\draw[-] (u5) to node[gap,inner sep=2pt] {\footnotesize$0$} (u4);
\end{tikzpicture}
};

\node at (8,-1.25) {$\det(\RR_7)$};
\node at (8,0) {
\begin{tikzpicture}[scale=0.5]
\node (u1) [circle,fill=black,inner sep=1pt] at (0,1.5) {};
\node (u1) [circle,fill=black,inner sep=1pt] at (0,1.5) {};
\node (u2) [circle,fill=black,inner sep=1pt] at (2.5,3) {};
\node (u3) [circle,fill=black,inner sep=1pt] at (2.5,1.5) {};
\node (u4) [circle,fill=black,inner sep=1pt] at (5,1.5) {};
\node (u5) [circle,fill=black,inner sep=1pt] at (2.5,0) {};
\draw[-] (u1) to node[gap,inner sep=2pt] {\footnotesize$1$} (u2);
\draw[-] (u2) to node[gap,inner sep=2pt] {\footnotesize$3$} (u3);
\draw[-] (u3) to node[gap,inner sep=2pt] {\footnotesize$1$} (u5);
\draw[-] (u1) to node[gap,inner sep=2pt] {\footnotesize$7$} (u5);
\draw[-] (u2) to node[gap,inner sep=2pt] {\footnotesize$2$} (u4);
\draw[-] (u5) to node[gap,inner sep=2pt] {\footnotesize$0$} (u4);
\end{tikzpicture}
};

\end{tikzpicture}
\vspace{-0.25cm}
\caption{Degrees of the bundles $\det(\RR_i)$ at curves $C\subset\Hilb{G}{\C^3}$.}
\label{exa:T13:degrees}
\end{center}
\end{figure}
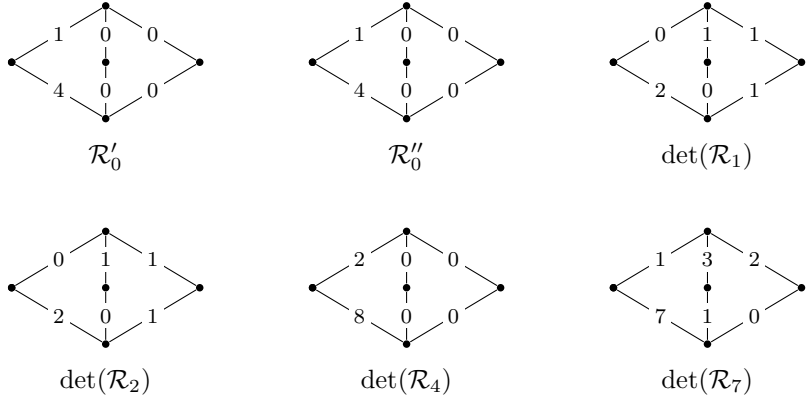

(e) $G$ has six nontrivial conjugacy classes, four of age 1 and two of age 2. This means that $e(Y)=|\Irr G| = 7 = h^0 + h^2 + h^4 = 1 + \rk\Pic Y + 2$, thus $\rk\Pic Y=4$. Using the relations in $(d)$ we see that $\RR_0'$, $\RR_0''$, $\det(\RR_2)$ and $\det(\RR_7)$ form a basis of $\Pic Y$, thus their first Chern classes for a basis of $H^2(Y,\Z)$. The cookery from Reid's recipe (see \cite{C2}) defines the following virtual bundles:
\[
\FF_{E_1}:=(\RR_{0}'\otimes\RR_{0}'')\ominus(\det\RR_4\otimes\OO_Y) \hspace{1cm} \FF_{E_2}:=\det\RR_2\ominus\det\RR_1
\]
which give a basis of $H^4(Y,\Z)$.
\end{proof}


\end{document}